\documentclass{amsart}
\usepackage{amssymb}
\usepackage{enumerate}
\usepackage[normalem]{ulem}

\usepackage{color}

\newtheorem{thm}{Theorem}[section]
\newtheorem{prop}[thm]{Proposition}
\newtheorem{lem}[thm]{Lemma}
\newtheorem{fact}[thm]{Scholium}

\newtheorem{cor}[thm]{Corollary}

\begin{document}
	
\title{Cwikel-Solomyak estimates on tori and Euclidean spaces}

\author[]{F. Sukochev}
\address{School of Mathematics and Statistics, University of New South Wales, Kensington,  2052, Australia}
\email{f.sukochev@unsw.edu.au}

\author[]{D. Zanin}
\address{School of Mathematics and Statistics, University of New South Wales, Kensington,  2052, Australia}
\email{d.zanin@unsw.edu.au}

\begin{abstract}
We revise Cwikel-type estimate for the singular values of the operator $(1-\Delta_{\mathbb{T}^d})^{-\frac{d}{4}}M_f(1-\Delta_{\mathbb{T}^d})^{-\frac{d}{4}}$ on the torus $\mathbb{T}^d$ for the ideal $\mathcal{L}_{1,\infty}$, established by M.Z. Solomyak in even dimensions in \cite{Solomyak1995}, and extend it to odd dimensions.  We obtain a new result for (symmetrized) Cwikel type estimates for Laplacians on $\mathbb{R}^d$ for arbitrary positive integer $d.$
\end{abstract}

\dedicatory{Dedicated to the memory of M.Z. Solomyak}

\maketitle

\section{Introduction}

Estimates for the operator $M_fg(\nabla)$ (here, $M_f$ is a multiplication operator and $g(\nabla)$ is a function of the gradient) take their origin in the study of bound states\footnote{In physicists' parlance, the eigenfunctions corresponding to the negative eigenvalues of Schr\"odinger operators are called bound states.} of Schr\"odinger operators. The problem of describing  functions $f$ and $g$ such that $M_fg(\nabla)$ belongs to weak Schatten classes $\mathcal{L}_{p,\infty}$ was initially stated by Simon (see Conjecture 1 in \cite{Simon-TAMS} and also Chapter 4 in \cite{Simon-book}). The first important result  in this direction
is due to Cwikel \cite{Cwikel} (see also Theorem 6.5 in \cite{Birman-Solomyak-uspekhi}). It states that
$$\|M_fg(\nabla)\|_{p,\infty}\leq c_p\|f\|_p\|g\|_{p,\infty},\quad f\in L_p(\mathbb{R}^d),\quad g\in L_{p,\infty}(\mathbb{R}^d),\quad 2<p\leq\infty.$$
Here, weak Schatten quasi-norm on the left hand side is given by the formula
$$\|T\|_{p,\infty}=\sup_{k\geq0}(k+1)^{\frac1p}\mu(k,T),$$
where $(\mu(k,T))_{k\geq0}$ is the singular value sequence of the operator $T.$

We refer to estimates of this kind as {\it generic} Cwikel estimates (the function $g$ of the gradient is arbitrary). Generic Cwikel estimates were strengthened in  \cite{Weidl}  as follows
$$\|M_fg(\nabla)\|_{p,\infty}\leq c_p\|f\otimes g\|_{p,\infty},\quad f\otimes g\in L_{p,\infty}(\mathbb{R}^d\times\mathbb{R}^d),\quad 2<p\leq\infty.$$
In \cite{LeSZ}, a more general version of this estimate, suitable for noncommutative variables $f$ and $g,$ is proved. The setting used in \cite{LeSZ} comes from quantised calculus and is fit for treating the concrete problems in Non-commutative Geometry. In particular, Cwikel estimates in \cite{LeSZ} are extended to non-commutative Euclidean (Moyal) space and allow treatment of the magnetic Laplacian.

In various applications (both to Mathematical Physics and to Non-commutative Geometry), {\it specific} Cwikel-type estimates at the critical dimension are the primary interest. Here, {\it specific} means that we fix the function $g$ to be
$$g(t)=(1+|t|^2)^{-\frac{d}{2p}},\quad t\in\mathbb{R}^d,\quad p>0,$$
and by critical dimension we mean $p=2.$ Physicists would be even more happy to consider the function $g(t)=|t|^{-\frac{d}{p}},$ however the corresponding operator $M_fg(\nabla)$ in critical dimension is known to be unbounded (see e.g. proof of Proposition 7.4 in \cite{Simon-book}) and hence falls outside the scope of this paper.

The best known specific Cwikel-type estimates (for $\mathbb{R}^d$ as well as for the $d$-dimensional torus $\mathbb{T}^d$) may be found in the foundational paper by Solomyak \cite{Solomyak1995}. In \cite{Solomyak1995}, the estimates are not stated explicitly  and only the case of even dimension is treated. The paper \cite{Solomyak1995} is based on the long line of works by Birman and Solomyak with co-authors \cite{BS67}, \cite{BS70}, \cite{BS72}, \cite{R72}, \cite{Birman-Solomyak-10th} which is also partly motivated by studying discrete spectrum of Schr\"odinger operators. A general scheme of quasi-norm estimates for Cwikel-type operators hatched in those papers was adapted in subsequent papers of Solomyak \cite{Solomyak1998} and Shargorodsky \cite{SharPLMS} to the case of even dimension and appropriate Orlicz norms.

We prove the following specific symmetrized Cwikel-type estimate for the ideal $\mathcal{L}_{1,\infty}$ in the setting of a $d$-dimensional torus $\mathbb{T}^d.$ Theorem~\ref{solomyak cwikel estimate torus} below states explicitly the results of Solomyak \cite{Solomyak1995} and extends them to arbitrary dimension.

\begin{thm}\label{solomyak cwikel estimate torus} Let $d\in\mathbb{N}.$ Let $M(t)=t\log(e+t),$ $t>0.$ We have
\begin{equation}\label{solomyak main estimate}
\Big\|(1-\Delta_{\mathbb{T}^d})^{-\frac{d}{4}}M_f(1-\Delta_{\mathbb{T}^d})^{-\frac{d}{4}}\Big\|_{1,\infty}\leq c_d\|f\|_{L_M(\mathbb{T}^d)},\quad f\in L_M(\mathbb{T}^d).
\end{equation}
\end{thm}
Here, the Orlicz space $L_M(\mathbb{T}^d)$ is the famous space $L{\rm log}L(\mathbb{T}^d)$ introduced by Zygmund in 1928 (see Section 4.6 in \cite{BenSh}). 

It is interesting to compare the result of Theorem \ref{solomyak cwikel estimate torus} with Theorem 1.2 from a recent paper due to S.~Lord and the authors \cite{LSZ-last-kalton}. There, via tensor multipliers technique from Banach space theory, it is shown that if $f \in L_M(\mathbb{R}^d)$ then  
$$(1-\Delta_{\mathbb{R}^d})^{-\frac{d}{4}}M_f(1-\Delta_{\mathbb{R}^d})^{-\frac{d}{4}}\in\mathcal{M}_{1,\infty}(L_2(\mathbb{R}^d)), $$
where the (Dixmier-Macaev) ideal  $\mathcal{M}_{1,\infty}$, the submajorization closure of $\mathcal{L}_{1,\infty}$, is strictly larger  than  $\mathcal{L}_{1,\infty}$ (see e.g. \cite{CRSS}). In the current manuscript we propose a different approach to derive the stronger estimate in Theorem~\ref{solomyak cwikel estimate torus} for the smaller ideal $\mathcal{L}_{1,\infty}.$ Our approach is based on Solomyak's ideas from \cite{Solomyak1994}, \cite{Solomyak1995} employed there for the case of even dimension.

A question asked by G. Rozenblum (private correspondence) is whether it is possible to extend the result of Theorem \ref{solomyak cwikel estimate torus} to Euclidean space. We show there is a stark contrast between the Dixmier-Macaev ideal $\mathcal{M}_{1,\infty}$ and the weak Schatten-von Neumann ideal $\mathcal{L}_{1,\infty}$ cases. The statement of Theorem~\ref{solomyak cwikel estimate torus} is false if $\mathbb{T}^d$ is replaced by $\mathbb{R}^d$, for \emph{any} symmetric function space on $\mathbb{R}^d!$ This surprising fact is established in Theorem \ref{bad rd theorem} below.

The following theorem answers Rozenblum's question in the negative.

\begin{thm}\label{bad rd theorem} For every symmetric quasi-Banach function space $E,$ there exists $f\in E(\mathbb{R}^d)$ such that the inequality
\begin{equation}\label{optimistic cwikel estimate}
\Big\|(1-\Delta_{\mathbb{R}^d})^{-\frac{d}{4}}M_f(1-\Delta_{\mathbb{R}^d})^{-\frac{d}{4}}\Big\|_{1,\infty}\leq \|f\|_E
\end{equation}
fails.
\end{thm}

Our third main result derives an alternate estimate and is given in Theorem \ref{best available rd cwikel} below, which yields a suitable extension of Theorem \ref{solomyak cwikel estimate torus} to Euclidean space $\mathbb{R}^d.$ This estimate captures the known results in the literature concerning Euclidean space estimates for weak ideals in the critical case \cite{Solomyak1994, Solomyak1995}. It delivers the best (to date) Cwikel-type estimate on $\mathbb{R}^d$ for the case of the weak Schatten class $\mathcal{L}_{1,\infty}.$

\begin{thm}\label{best available rd cwikel} Let $d\in\mathbb{N}.$ Let $M(t)=t\log(e+t),$ $t>0,$ and let $f\in L_M(\mathbb{R}^d).$ We have
$$\Big\|(1-\Delta_{\mathbb{R}^d})^{-\frac{d}{4}}M_f(1-\Delta_{\mathbb{R}^d})^{-\frac{d}{4}}\Big\|_{1,\infty}\leq c_d\Big(\|f\|_{L_M(\mathbb{R}^d)}+\int_{\mathbb{R}^d}|f(s)|\log(1+|s|)ds\Big).$$ 
\end{thm}

It has been already proved in Section 2.5 in \cite{LSZ-last-kalton} that the operator featuring in Theorem \ref{best available rd cwikel} is bounded whenever $f$ belongs to the Lorentz space $\Lambda_1(\mathbb{R}^d).$ However, the Lorentz space $\Lambda_1(\mathbb{R}^d)$ and the Orlicz space $L_M(\mathbb{R}^d)$ are known to coincide (see e.g. a similar assertion in Lemma 4.6.2 in \cite{BenSh}).

In the special case when $d=2$, $f\geq0,$ and $g(t)=|t|^{-1}$ the expression featuring on the right hand side of the inequality in Theorem \ref{best available rd cwikel} can be glimpsed in \cite{SharPLMS}, where it was used for obtaining sharp estimates for the number of negative eigenvalues of the Schr\"odinger operator. Note however, that Cwikel-type estimates were not considered in \cite{SharPLMS}.

The proof of Theorem \ref{best available rd cwikel} reveals the fundamental fact: conformal invariance of Cwikel-type estimates. In the pre-critical case, this idea can be traced back to \cite{GGM}. Frank \cite{Frank} investigated conformal invariance (for Rumin inequality which happened to be equivalent to Cwikel-type estimate) in the pre-critical case. We prove the invariance of Cwikel-type estimate in the crucial case with respect to inversion (essentially, the only non-linear conformal transform for $d>2$).

Theorem \ref{best available rd cwikel} is new for dimensions $d\neq 2.$ For $d=2$ it can be deduced with modest effort from the results of Solomyak \cite{Solomyak1994} and Shargorodsky \cite{SharPLMS}. The proof is presented in Appendix \ref{frank section} and is due to Professor Frank.

In Appendix A, we present an alternative description of the quantity standing on the right hand side of Theorem \ref{best available rd cwikel} (see Proposition \ref{equivalence prop}).

\subsection{Strategy of the proof} 

Our approach to the proof of Theorem 1.1 is based on Sobolev embedding theorem and follows the pattern elaborated in the cited papers by Birman-Solomyak,  with crucial improvement by Solomyak \cite{Solomyak1994, Solomyak1995}. 

One should note that already the boundeness (in the uniform norm) of the operator 
$$(1-\Delta_{\mathbb{T}^d})^{-\frac{d}{4}}M_f(1-\Delta_{\mathbb{T}^d})^{-\frac{d}{4}}$$
is non-trivial (for an unbounded measurable function $f$ on $\mathbb{T}^d$). Indeed, the estimate on the uniform norm of this operator is equivalent (see e.g. Theorem 2.3 in \cite{LSZ-last-kalton}) to the critical case of Sobolev embedding theorem. Trudinger \cite{Trudinger} proved that the Sobolev space $W^{\frac{d}{2},2}(\mathbb{T}^d)$ embeds into the Orlicz space $\exp(L_2)(\mathbb{T}^d)$ (see also Theorem \ref{ozawa estimate} below).

In Section \ref{solomyak thm torus} below we restate Theorem \ref{solomyak cwikel estimate torus} as:
$$\Big\|M_{f^{\frac12}}(1-\Delta_{\mathbb{T}^d})^{-\frac{d}{4}}\Big\|_{2,\infty}\leq c_d\|f\|_{L_M(\mathbb{T}^d)}^{\frac12},\quad 0\leq f\in L_M(\mathbb{T}^d).$$
Note that $(1-\Delta_{\mathbb{T}^d})^{-\frac{d}{4}}$ sends $L_2(\mathbb{T}^d)$ into Sobolev space $W^{\frac{d}{2},2}(\mathbb{T}^d).$ It is easily verified that the identity mapping ${\rm id}:W^{\frac{d}{2},2}(\mathbb{T}^d)\to L_2(\mathbb{T}^d)$ is a compact operator. Hence, at least for a bounded $f,$ the multiplication mapping $M_{f^{\frac12}}$ from the $W^{\frac{d}{2},2}(\mathbb{T}^d)$ into $L_2(\mathbb{T}^d)$ is also compact. 
We adopt Solomyak's viewpoint on Theorem \ref{solomyak cwikel estimate torus} as an estimate of approximation numbers of the operator $M_{f^{\frac12}}$ from the $W^{\frac{d}{2},2}(\mathbb{T}^d)$ into $L_2(\mathbb{T}^d)$ (this viewpoint is made clear in Lemma \ref{solomyak intermediate} below). 
Solomyak employed the methods developed by Birman and Solomyak presented e.g. in the book \cite{Birman-Solomyak-10th} (see Theorems 1.1-1.4 there and subsequent explanations). 
The key tools in our proof are the {\it homogeneous} Sobolev inequality on the cube, Theorem \ref{comparison thm}, and
 Besicovitch Covering Lemma \ref{besicovitch lemma}. The usage of coverings instead of previously used partitions, in constructing approximating finite rank operators was pioneered by Rozenblum, see also the comments preceding the proof of Theorem \ref{solomyak cover}. The crucial importance of Theorem \ref{comparison thm} becomes apparent in the proof of Lemma \ref{scaled holder lemma}. Besicovitch Covering Lemma is then used to choose a linear operator of a given rank $n$ which approximates $M_{f^{\frac12}}$ with required accuracy.

Trudinger's result referenced above was latter strengthened by Hansson, Brezis and Wainger, Cwikel and Pustylnik and subjected to further analysis in \cite{RSZ}, where it is proposed to replace norm-estimates with distributional ones. In critical dimensions this approach allows to compute the uniform norm of Cwikel operator and  becomes an indispensable tool in the proof of Theorem \ref{bad rd theorem}.

The technique in the proof of Theorem \ref{best available rd cwikel} relies on the inversion trick (attributed in \cite{SharPLMS} to \cite{GrNad}). This technique allows to compare Cwikel operators
$$(1-\Delta_{\mathbb{R}^d})^{-\frac{d}{4}}M_f(1-\Delta_{\mathbb{R}^d})^{-\frac{d}{4}},\quad (1-\Delta_{\mathbb{R}^d})^{-\frac{d}{4}}M_{Vf}(1-\Delta_{\mathbb{R}^d})^{-\frac{d}{4}},$$
where the function $Vf$ is defined before the Lemma \ref{urd def lemma}. It is crucial that, whenever the function $f$ is supported outside of the unit ball, the function $Vf$ is supported inside the unit ball. This idea allows to reduce the problem to the case when $f$ is supported in the unit ball, that is, essentially, to Cwikel-type estimates on a torus $\mathbb{T}^d.$ To the best of our knowledge, this is the first usage of the inversion trick in the studies of Cwikel-type estimates.

{\bf Acknowledgement}: The authors thank Professors Trudinger and Valdinoci for useful discussions about the Sobolev embedding theorem and Professor Rozenblum for for his interest in the paper and for discussions which led to numerous improvements (both mathematical and historical) in the exposition. We thank Professor Frank for for communicating to us the result presented in Appendix \ref{frank section} and for drawing our attention to the reference \cite{Frank}. We thank Galina Levitina for detailed reading and commenting on the manuscript.

\section{Preliminaries}

Everywhere below, constants $c_{x,y}$ depend only on the choice of $x,y.$ Exact values of this constant may change from line to line.

Everywhere below an integral without explicitly written measure is assumed to be taken with respect to the Lebesgue measure.

\subsection{Symmetric function spaces}

Let $(\Omega,\nu)$ be a measure space. Let $S(\Omega,\nu)$ be the collection of all $\nu$-measurable functions on $\Omega$ such that, for some $n\in\mathbb{N},$ the function $|f|\chi_{\{|f|>n\}}$ is supported on a set of finite measure. For every $f\in S(\Omega,\nu),$ the distribution function
$$t\to \nu(\{|f|>t\}),\quad t>0,$$
is finite for all sufficiently large $t.$  For every $f\in S(\Omega,\nu)$ one can define the notion of decreasing rearrangement of $f$ (denoted by $\mu(f)$). This is a positive decreasing function on $\mathbb{R}_+$ equimeasurable with $|f|.$ 

Let $E(\Omega,\nu)\subset S(\Omega,\nu)$ and let $\|\cdot\|_E$ be Banach norm on $E(\Omega,\nu)$ such that
\begin{enumerate}[{\rm (1)}]
\item if $f\in E(\Omega,\nu)$ and $g\in S(\Omega,\nu)$ be such $|g|\leq|f|,$ then $g\in E(\Omega,\nu)$ and $\|g\|_E\leq\|f\|_E;$
\item if $f\in E(\Omega,\nu)$ and $g\in S(\Omega,\nu)$ be such $\mu(g)=\mu(f),$ then $g\in E(\Omega,\nu)$ and $\|g\|_E=\|f\|_E;$
\end{enumerate}
We say that $(E(\Omega,\nu),\|\cdot\|_E)$ (or simply $E$) is a symmetric Banach function space (symmetric space, for brevity).

If $\Omega=\mathbb{R}_+,$ then the function
$$t\to\|\chi_{(0,t)}\|_E,\quad t>0,$$
is called the {\it fundamental} function of $E.$ Similar definition is available when $\Omega$ is an interval or an arbitrary measure space.  The concrete examples of measure spaces $(\Omega, \nu)$ considered in this paper are $\mathbb{T}^d$ (equipped with the normalised Haar measure $m$), $\mathbb{R}_+,$ $\mathbb{R}^d$ (equipped with Lebesgue measure $m$), their measurable subsets and compact $d$-dimensional Riemannian manifolds $(X,g).$

Among concrete symmetric spaces used in this paper are $L_p$-spaces and Orlicz spaces. Given an even convex function $M$ on $\mathbb{R}$ such that $M(0)=0,$ Orlicz space $L_M(\Omega,\nu)$ is defined by setting
$$L_M(\Omega,\nu)=\Big\{f\in S(\Omega,\nu):\ M(\lambda|f|)\in L_1(\Omega,\nu)\mbox{ for some }\lambda>0\Big\}.$$
We equip it with a norm
$$\|f\|_{L_M}=\inf\Big\{\lambda>0:\ \Big\|M(\frac{|f|}{\lambda})\Big\|_1\leq 1\Big\}.$$
We refer the reader to \cite{KrasRut} for further information about Orlicz spaces.

For a particular function $M(t)=t\log(e+t),$ $t>0,$ we have $f\in L_M(\mathbb{R}^d)$ if and only if $\mu(f)\chi_{(0,1)}\in L_M(0,1)$ and $f\in L_1(\mathbb{R}^d).$

We also need a definition of dilation operator $\sigma_u,$ $u>0,$ which acts on $S(\mathbb{R},m)$ (or on $S(\mathbb{R}^d,m)$) by the formula
$$(\sigma_uf)(t)=f(\frac{t}{u}),\quad f\in S(\mathbb{R},m).$$
It is sometimes convenient to consider dilations of functions which are defined {\it a priori} only on some subset (typically, an interval or a cube) of $\mathbb{R}$ or $\mathbb{R}^d.$ In this case, we first extend $f$ to a function on $\mathbb{R}$ (or $\mathbb{R}^d$) by setting $f=0$ outside of the initial domain of $f.$

\subsection{Trace ideals}
The following material is standard; for more details we refer the reader to \cite{LSZ-book,Simon-book}.
Let $H$ be a complex separable infinite dimensional Hilbert space, and let $B(H)$ denote the set of all bounded operators on $H$, and let $K(H)$ denote the ideal of compact operators on $H.$ Given $T\in K(H),$ the sequence of singular values $\mu(T) = \{\mu(k,T)\}_{k=0}^\infty$ is defined as:
\begin{equation*}
\mu(k,T) = \inf\{\|T-R\|_{\infty}:\quad \mathrm{rank}(R) \leq k\}.
\end{equation*}
It is often convenient to identify the sequence $(\mu(k,T))_{k\geq0}$ with a step function $\sum_{k\geq0}\mu(k,T)\chi_{(k,k+1)}.$

Let $p \in (0,\infty).$ The weak Schatten class $\mathcal{L}_{p,\infty}$ is the set of operators $T$ such that $\mu(T)$ is in the weak $L_p$-space $l_{p,\infty}$, with the quasi-norm:
\begin{equation*}
\|T\|_{p,\infty} = \sup_{k\geq 0} (k+1)^{\frac1p}\mu(k,T) < \infty.
\end{equation*}
Obviously, $\mathcal{L}_{p,\infty}$ is an ideal in $B(H).$ We also have the following form
of H\"older's inequality,
\begin{equation}\label{weak holder}
\|TS\|_{r,\infty} \leq c_{p,q}\|T\|_{p,\infty}\|S\|_{q,\infty}
\end{equation}
where $\frac{1}{r}=\frac{1}{p}+\frac{1}{q}$, for some constant $c_{p,q}$. Indeed, this follows from the definition of these quasi-norms and the inequality (see e.g. \cite[Proposition 1.6]{Fack1982}, \cite[Corollary 2.2]{GohbergKrein})
$$\mu(2n,TS)\leq \mu(n,T)\mu(n,S),\quad n\geq 0.$$

The ideal of particular interest is $\mathcal{L}_{1,\infty}$, and we are concerned with traces on this ideal. For more details, see \cite[Section 5.7]{LSZ-book} and \cite{SSUZ2015}. A linear functional $\varphi:\mathcal{L}_{1,\infty}\to\mathbb{C}$ is called a trace if it is unitarily invariant. That is, for all unitary operators $U$ and for all $T\in\mathcal{L}_{1,\infty}$ we have that $\varphi(U^{\ast}TU) = \varphi(T)$. It follows that for all bounded operators $B$ we have $\varphi(BT)=\varphi(TB).$  

Every trace $\varphi:\mathcal{L}_{1,\infty}\to\mathbb{C}$ vanishes on the ideal of finite rank operators (such traces are called singular). In fact, $\varphi$ vanishes on the ideal $\mathcal{L}_1$ (see \cite{DFWW} or \cite{LSZ-book}). For the state of the art in the theory of singular traces and their applications in Non-commutative Geometry, we refer the reader to the survey \cite{LSZ-survey}.

\subsection{Sobolev spaces on cubes}

Let $m\in\mathbb{Z}_+.$ For every cube $\Pi,$ we define Sobolev space $W^{m,2}(\Pi)$ as follows
$$W^{m,2}(\Pi)=\big\{ u\in L_2(\Pi):\ \nabla^{\alpha}u\in L_2(\Pi),\quad |\alpha|_1\leq m  \big\}.$$
Here, $\nabla^{\alpha}f$ is understood as a distributional derivative (with respect to the space of test functions $C^{\infty}_c({\rm int}(\Pi))$ on the interior ${\rm int}(\Pi)$ of the cube). We equip $W^{m,2}(\Pi)$ with the (non-homogeneous) Sobolev norm by the formula (see p.44 in \cite{Adams}):
$$\|u\|_{W^{m,2}(\Pi)}^2=\sum_{|\alpha|_1\leq m}\|\nabla^{\alpha}u\|_{L_2(\Pi)}^2$$
for every $u\in W^{m,2}(\Pi).$ It is a standard fact (see e.g. Theorem 3.5 in \cite{Adams}) that  $(W^{m,2}(\Pi),\|\cdot\|_{W^{m,2}(\Pi)})$ is a Hilbert space.

Let $s>0$ and let $m=\lfloor s\rfloor.$ If $s\neq m,$ then we define Sobolev space $W^{s,2}(\Pi)$ as follows
$$W^{s,2}(\Pi)=\Big\{ u\in W^{m,2}(\Pi): \int_{\Pi}\int_{\Pi}\frac{|(\nabla^{\alpha}u)(x)-(\nabla^{\alpha}u)(y)|^2}{|x-y|_2^{d+2(s-m)}}dxdy<\infty \Big\}.$$
We equip $W^{s,2}(\Pi)$ with the (non-homogeneous) Sobolev norm by the formula (see Theorem 7.48 in \cite{Adams}):
$$\|u\|_{W^{s,2}(\Pi)}^2=\|u\|_{W^{m,2}(\Pi)}^2+\sum_{|\alpha|_1\leq m}\int_{\Pi}\int_{\Pi}\frac{|(\nabla^{\alpha}u)(x)-(\nabla^{\alpha}u)(y)|^2}{|x-y|_2^{d+2(s-m)}}dxdy$$
for every $u\in W^{s,2}(\Pi).$ It is known  that $(W^{s,2}(\Pi),\|\cdot\|_{W^{s,2}(\Pi)})$ is a Hilbert space (see e.g. p.205 and Theorem 7.48 in \cite{Adams} for the proof of completeness; the parallelogram identity may be directly verified).

\subsection{Sobolev spaces on $\mathbb{R}^d$ and on $\mathbb{T}^d$}

Recall that Sobolev space $W^{s,2}(\mathbb{R}^d)$ admits an easier description (see e.g. Theorem 7.63 in \cite{Adams}):
$$W^{s,2}(\mathbb{R}^d)=\big\{ u\in L_2(\mathbb{R}^d):\ (1-\Delta_{\mathbb{R}^d})^{\frac{s}{2}}u\in L_2(\mathbb{R}^d)\big\}$$
with an equivalent norm
$$\|u\|_{W^{s,2}(\mathbb{R}^d)}=\|(1-\Delta_{\mathbb{R}^d})^{\frac{s}{2}}u\|_2,\quad u\in W^{s,2}(\mathbb{R}^d).$$
Here, $\Delta_{\mathbb{R}^d}$ is the Laplace operator on $\mathbb{R}^d.$

We also need the notion of Sobolev space on the torus:
$$W^{s,2}(\mathbb{T}^d)=\big\{ u\in L_2(\mathbb{T}^d):\ (1-\Delta_{\mathbb{T}^d})^{\frac{s}{2}}u\in L_2(\mathbb{T}^d)\big\}$$
with the norm
$$\|u\|_{W^{s,2}(\mathbb{T}^d)}=\|(1-\Delta_{\mathbb{T}^d})^{\frac{s}{2}}u\|_2,\quad u\in W^{s,2}(\mathbb{T}^d).$$
Here, $\Delta_{\mathbb{T}^d}$ is the Laplace operator on $\mathbb{T}^d.$

\subsection{Comparison: Sobolev spaces on cube vs Sobolev spaces on $\mathbb{R}^d$ and on $\mathbb{T}^d$}

The following result can be found in \cite{Adams} (e.g., by combining Theorems 7.41 and 7.48 there).

\begin{thm}\label{extension theorem} Let $\Pi=[-\pi,\pi]^d$ and let $s>0.$ For every $u\in W^{s,2}(\Pi),$ there exists $u_{\mathbb{R}^d}\in W^{s,2}(\mathbb{R}^d)$ such that $u_{\mathbb{R}^d}|_{\Pi}=u$ and such that $\|u_{\mathbb{R}^d}\|_{W^{s,2}(\mathbb{R}^d)}\leq c_{s,d}\|u\|_{W^{s,2}(\Pi)}.$	
\end{thm}

Let $\mathbb{T}^d$ be $d$-dimensional torus. We identify $\mathbb{T}^d$ with the cube $[-\pi,\pi]^d$ whose opposite faces are glued. We equip $\mathbb{T}^d$ with the normalised Haar measure $m.$

\begin{thm}\label{sobolev torus vs sobolev cube} If $\Pi=[-\pi,\pi]^d,$ then ${\rm id}:W^{s,2}(\mathbb{T}^d)\to W^{s,2}(\Pi)$ is a bounded mapping for every $s>0.$
\end{thm}
\begin{proof} Let ${\rm per}:L_2(\mathbb{T}^d)\to L_2^{{\rm loc}}(\mathbb{R}^d)$ be the extension by periodicity. Let $\phi$ be a Schwartz function on $\mathbb{R}^d$ such that $\phi=1$ on $\Pi.$
	
The mapping
$$A:u\to \phi\cdot{\rm per}(u),\quad u\in W^{m,2}(\mathbb{T}^d),$$
is well defined and bounded from $W^{m,2}(\mathbb{T}^d)$ into $W^{m,2}(\mathbb{R}^d)$ for every $m\in\mathbb{Z}_+.$ By complex interpolation (see Theorem 7.65 in \cite{Adams}), $A$ is a bounded mapping from $W^{s,2}(\mathbb{T}^d)$ into $W^{s,2}(\mathbb{R}^d)$ for every $s>0.$

Since $Au|_{\Pi}=u,$ it follows that
$$\|u\|_{W^{s,2}(\Pi)}=\|Au|_{\Pi}\|_{W^{s,2}(\Pi)}\leq \|Au\|_{W^{s,2}(\mathbb{R}^d)}\leq$$
$$\leq\|A\|_{W^{s,2}(\mathbb{T}^d)\to W^{s,2}(\mathbb{R}^d)}\|u\|_{W^{s,2}(\mathbb{T}^d)},\quad u\in W^{m,2}(\mathbb{T}^d),\quad s>0.$$
\end{proof}

\subsection{Sobolev embedding theorem for $s=\frac{d}{2}$}

The following result is the well-known Moser-Trudinger inequality \cite{Trudinger}. Similar result for $\mathbb{R}^d$ was proved in \cite[Lemma 2.2]{LSZ-last-kalton} (based on results of \cite{Ozawa}).

In what follows, $\exp(L_2)$ denotes the Orlicz space associated with the Orlicz function $t\to e^{t^2}-1,$ $t>0.$

\begin{thm}\label{ozawa estimate} Let $d\in\mathbb{N}$ and let $\Pi=[-\pi,\pi]^d.$ If $u\in W^{\frac{d}{2},2}(\Pi),$ then
$$\|u\|_{\exp(L_2)(\Pi)}\leq c_d\|u\|_{W^{\frac{d}{2},2}(\Pi)}.$$
\end{thm}

\subsection{Homogeneous semi-norms on Sobolev spaces}

In what follows, we need the notion of the homogeneous Sobolev semi-norm: for $s=m\in\mathbb{Z}_+,$ it is defined by the formula
$$\|u\|_{W^{m,2}_{{\rm hom}}(\Pi)}^2=\sum_{|\alpha|_1=m}\|\nabla^{\alpha}u\|_{L_2(\Pi)}^2.$$
For $s\notin\mathbb{Z}_+,$ $m=\lfloor s\rfloor,$ it is defined by the formula
$$\|u\|_{W^{s,2}_{{\rm hom}}(\Pi)}^2=\sum_{|\alpha|_1=m}\int_{\Pi}\int_{\Pi}\frac{|(\nabla^{\alpha}u)(x)-(\nabla^{\alpha}u)(y)|^2}{|x-y|_2^{d+2(s-m)}}dxdy.$$
It is immediate that
$$\|u\|_{W^{s,2}_{{\rm hom}}(\Pi)}\leq \|u\|_{W^{s,2}(\Pi)},\quad u\in W^{s,2}(\Pi).$$

For integer $s,$ the following assertion is Theorem 1.1.16 in \cite{Mazya}. In \cite{Solomyak1994}, Solomyak used it (for even $d$ and for $s=\frac{d}{2}$) without a proof or reference. The proof below is provided to us by G. Rozenblum (according to him, this result is folklore in St Petersburg school). Rozenblum's proof is simpler than our original argument and is included here with his kind permission.

\begin{thm}\label{comparison thm} Let $d\in\mathbb{N}$ and let $\Pi=[-\pi,\pi]^d.$ If $u\in W^{s,2}(\Pi),$ $s>0,$ is orthogonal (in $L_2(\Pi)$) to every polynomial of degree strictly less than $s,$ then
$$\|u\|_{W^{s,2}(\Pi)}\leq c_{s,d}\|u\|_{W^{s,2}_{{\rm hom}}(\Pi)}.$$
\end{thm}
\begin{proof} We only prove the assertion for non-integer $s.$ Set $m=\lfloor s\rfloor.$
	
	Assume the contrary and choose a sequence $(u_k)_{k\geq0}\subset W^{s,2}(\Pi)$ such that
\begin{enumerate}
\item $\|u_k\|_{W^{m,2}(\Pi)}=1$ for every $k\geq0;$
\item $\|u_k\|_{W^{s,2}_{{\rm hom}}(\Pi)}\to0$ as $k\to\infty;$
\item $\langle u_k,p\rangle_{L_2(\Pi)}=0$ for every $k\geq0$ and for every polynomial $p$ of degree $m.$  
\end{enumerate}	
In particular, for every $\alpha$ with $|\alpha|_1=m,$ we have
\begin{equation}\label{rl eq1}
\|\nabla^{\alpha}u_k\|_{W^{s-m,2}_{{\rm hom}}(\Pi)}\to0,\quad k\to\infty.
\end{equation}

It is crucial that $W^{s,2}(\Pi)$ is compactly embedded into $W^{m,2}(\Pi)$ (this fundamental fact is available, e.g. in Theorem 3.27 in \cite{McLean}). Passing to a subsequence, if needed, we may assume that $u_k\to u$ in $W^{m,2}(\Pi).$

For every $\alpha$ with $|\alpha|_1=m,$ $\nabla^{\alpha}u_k\to \nabla^{\alpha}u$ in $L_2(\Pi).$ Passing to a subsequence, if needed, we may assume that $\nabla^{\alpha}u_k\to \nabla^{\alpha}u$ almost everywhere.

Fix $\alpha$ with $|\alpha|_1=m.$ Set
$$v_k(x,y)=\frac{(\nabla^{\alpha}u_k)(x)-(\nabla^{\alpha}u_k)(y)}{|x-y|_2^{d+2(s-m)}},\quad x,y\in\Pi,$$
$$v(x,y)=\frac{(\nabla^{\alpha}u)(x)-(\nabla^{\alpha}u)(y)}{|x-y|_2^{d+2(s-m)}},\quad x,y\in\Pi.$$
It follows that $v_k\to v$ almost everywhere. On the other hand, \eqref{rl eq1} means that $v_k\to0$ in $L_2(\Pi\times\Pi).$ It follows that $v=0.$ Equivalently, $\nabla^{\alpha}u$ is a constant. 

Since $\nabla^{\alpha}u$ is a constant for every $\alpha$ with $|\alpha|_1=m,$ it follows that $u$ is a polynomial of degree $m$ (or less). Let $p$ be any polynomial of degree $m$ (or less). Since the mapping
$$f\to \langle f,p\rangle_{L_2(\Pi)},\quad f\in W^{m,2}(\Pi),$$
is a continuous linear functional on $W^{m,2}(\Pi),$ it follows that
$$\langle u_k,p\rangle_{L_2(\Pi)}\to \langle u,p\rangle_{L_2(\Pi)},\quad k\to\infty.$$
On the other hand, the choice of $u_k$ is such that
$$\langle u_k,p\rangle_{L_2(\Pi)}=0,\quad k\geq0.$$
Thus,
$$\langle u,p\rangle_{L_2(\Pi)}=0$$
for every  polynomial $p$ of degree $m$ (or less). Since $u$ itself is a polynomial of degree $m,$ it follows that $u=0.$ Therefore, $u_k\to0$ in $W^{m,2}(\Pi),$ which contradicts the choice $\|u_k\|_{W^{m,2}(\Pi)}=1$ for every $k\geq0.$   
\end{proof}

\section{Solomyak-type theorem on coverings}

Formally, Theorem \ref{solomyak cover} below is new. However, its result can be extracted from \cite[pp.258-260]{Solomyak1994}.

Recall that the torus $\mathbb{T}^d$ is equipped with a normalised Haar measure. For an Orlicz function $M$ and for $f\in L_M(\mathbb{T}^d),$ we set
$$J_f^M(A)=m(A)\Big\|\sigma_{\frac1{m(A)}}\mu(f|_A)\Big\|_{L_M},\quad A\subset\mathbb{T}^d,\quad m(A)>0.$$
The definition above is technically simpler (though, eventually, equivalent) than the one given in \cite{Solomyak1994} (see formulae (4) and (13) there).

In this and subsequent sections we view  torus $\mathbb{T}^d$ as a Cartesian product of $d$ circles, and a cube in $\mathbb{T}^d$ is defined as a Cartesian product of arcs of equal length.

\begin{thm}\label{solomyak cover} Let $L_M$ be a separable Orlicz space on $(0,1).$ For every $f\in L_M(\mathbb{T}^d)$ and for every $n\in\mathbb{N},$ there exist $m(n)\leq c_dn$ and a collection $(\Pi_k)_{k=1}^{m(n)}$ of cubes in $\mathbb{T}^d$ such that
\begin{enumerate}[{\rm (i)}]
\item\label{cova} each point in $\mathbb{T}^d$ belongs to at least one of $\Pi_k,$ $1\leq k\leq m(n);$
\item\label{covb} each point in $\mathbb{T}^d$ belongs to at most $c_d$ of $\Pi_k,$ $1\leq k\leq m(n);$
\item\label{covc} for every $1\leq k\leq m(n),$ we have $J_f^M(\Pi_k)=\frac1n\|f\|_{L_M}.$
\end{enumerate}
\end{thm}

Lemma below manifests the fact that every Orlicz space is distributionally concave (see \cite{ASW} for detailed discussion of this notion). The usage of this concept distinguishes our proof from that in \cite{Solomyak1994}.

We write $\bigoplus_{i\in\mathbb{I}}x_i$ for the disjoint sum of the functions $(x_i)_{i\in\mathbb{I}}.$

\begin{lem}\label{orlicz dist concave lemma} Let $M$ be an Orlicz function and let $L_M$ be respective Orlicz space either on $(0,1)$ or on $(0,\infty).$ We have
$$4\Big\|\bigoplus_{k\geq1}\sigma_{\lambda_k}f_k\Big\|_{L_M}\geq \sum_{k\geq1}\lambda_k\|f_k\|_{L_M}$$
for every sequence $(f_k)_{k\geq1}\subset L_M$ and for every scalar sequence $(\lambda_k)_{k\geq1}\subset(0,1)$ such that $\sum_{k\geq1}\lambda_k=1.$
\end{lem}
\begin{proof} For definiteness, we consider the spaces on $(0,\infty).$ Let $N$ be the complementary Orlicz function (see e.g. \cite{KrasRut}). We have (see equation (9.24) in \cite{KrasRut}) that
$$\|x\|_{L_M}\leq\sup_{\|y\|_{L_N}\leq 1}|\langle x,y\rangle|\leq 2\|x\|_{L_M}.$$
Here,
$$\langle x,y\rangle=\int_0^{\infty}x(s)y(s)ds,\quad x\in L_M(0,\infty),\quad y\in L_N(0,\infty).$$

Choose $g_k\in L_N$ such that $\|g_k\|_{L_N}\leq 1$ and such that
$$\langle f_k,g_k\rangle\geq\frac12\|f_k\|_{L_M}.$$
We have
\begin{align*}
\sum_{k\geq1}\lambda_k\|f_k\|_{L_M}& \leq 2\sum_{k\geq1}\lambda_k\langle f_k,g_k\rangle=2\sum_{k\geq1}\langle \sigma_{\lambda_k}f_k,\sigma_{\lambda_k}g_k\rangle
\\ &=2\Big\langle \bigoplus_{k\geq1}\sigma_{\lambda_k}f_k,\bigoplus_{k\geq1}\sigma_{\lambda_k}g_k\Big\rangle\leq 4\Big\|\bigoplus_{k\geq1}\sigma_{\lambda_k}f_k\Big\|_{L_M}\Big\|\bigoplus_{k\geq1}\sigma_{\lambda_k}g_k\Big\|_{L_N}.
\end{align*}

Since $\|g_k\|_{L_N}\leq1,$ it follows that $\|N(g_k)\|_1\leq1.$ Thus,
$$\Big\|N\Big(\bigoplus_{k\geq1}\sigma_{\lambda_k}g_k\Big)\Big\|_1=\sum_{k\geq1}\Big\|N\Big(\sigma_{\lambda_k}g_k\Big)\Big\|_1=\sum_{k\geq1}\lambda_k\|N(g_k)\|_1\leq1$$
and
$$\Big\|\bigoplus_{k\geq1}\sigma_{\lambda_k}g_k\Big\|_{L_N}\leq 1.$$
A combination of these estimates yields the assertion.
\end{proof}

Next lemma delivers subadditivity of the functional $J_f^M$ and is an easy consequence of Lemma \ref{orlicz dist concave lemma}.

\begin{lem}\label{lorentz lemma} Let $M$ and $f$ be as in Theorem \ref{solomyak cover}. If $(A_k)_{k=0}^n$ is an arbitary measurable partition of $\mathbb{T}^d,$ then
$$\sum_{k=0}^nJ_f^M(A_k)\leq 4\|f\|_{L_M}.$$
\end{lem}
\begin{proof} Set $\lambda_k=m(A_k)$, $1\leq k\leq n$ so that $\sum _{k=1}^n \lambda_k=1$ and let
$$f_k=\sigma_{\lambda_k^{-1}}\mu(f|_{A_k}),\quad 1\leq k\leq n.$$
It is immediate that
$$\mu(f)=\mu\Big(\bigoplus_{k=1}^n\sigma_{\lambda_k}f_k\Big).$$
By Lemma \ref{orlicz dist concave lemma}, we have
$$4\|f\|_{L_M}\geq \sum_{k=1}^n\lambda_k\|f_k\|_{L_M}=\sum_{k=1}^nJ_f^M(A_k).$$
\end{proof}

We equip the Boolean algebra of Lebesgue measurable sets in $\mathbb{T}^d$ with the usual metric
$${\rm dist}(A_1,A_2)=m(A_1\bigtriangleup A_2),\quad A_1, A_2\subset \mathbb{T}^d.$$

For a given $f\in L_M(\mathbb{T}^d)$ define a function $F_f:[0,1]\to\mathbb{R}_+$ by setting
$$F_f(t)=2\|\mu(f)\chi_{(0,t)}\|_{L_M}+2t^{\frac12}\|f\|_{L_M}+4t^{\frac12}\Big\|\sigma_{\frac1{2t^{\frac12}}}\mu(f)\Big\|_{L_M},\quad t\in[0,1].$$

The following assertion slightly improves Lemma 4 in \cite{Solomyak1994} and adjusts it to the case of $\mathbb{T}^d$.

\begin{lem}\label{j continuity} Let $L_M$ be a separable Orlicz space on $(0,1).$ For every $f\in L_M(\mathbb{T}^d),$ the mapping $A\to J_f^M(A)$ is continuous with respect to the metric ${\rm dist}.$ More precisely, for all measurable sets $A_1,A_2\subset\mathbb{T}^d,$ we have
$$|J_f^M(A_1)-J_f^M(A_2)|\leq F_f({\rm dist}(A_1,A_2)).$$
\end{lem}
\begin{proof} Fix $\epsilon>0$ and suppose $m(A_1\bigtriangleup A_2)<\epsilon^2.$ We consider the two logically possible cases separately.
	
{\bf Case 1:} Let $m(A_1)>\epsilon$ and $m(A_2)>\epsilon.$ Set $A_3=A_1\cup A_2.$ Note that
$$m(A_1)\leq m(A_3)\leq (1+\epsilon)m(A_1),\quad m(A_2)\leq m(A_3)\leq (1+\epsilon)m(A_2).$$
	
By triangle inequality, we have
\begin{align*}
J_f^M(A_3)&=m(A_3)\|\sigma_{\frac1{m(A_3)}}\mu(f|_{A_3})\|_{L_M}\\&
\leq m(A_3)\|\sigma_{\frac1{m(A_3)}}\mu(f|_{A_2\backslash A_1})\|_{L_M}+m(A_3)\|\sigma_{\frac1{m(A_3)}}\mu(f|_{A_1})\|_{L_M}.
\end{align*}

Obviously,
$$m(A_3)\|\sigma_{\frac1{m(A_3)}}\mu(f|_{A_2\backslash A_1})\|_{L_M}\leq \|f|_{A_2\backslash A_1}\|_{L_M}\leq \|\mu(f)\chi_{(0,\epsilon^2)}\|_{L_M}$$
and
$$m(A_3)\|\sigma_{\frac1{m(A_3)}}\mu(f|_{A_1})\|_{L_M}\leq m(A_3)\|\sigma_{\frac1{m(A_1)}}\mu(f|_{A_1})\|_{L_M}=\frac{m(A_3)}{m(A_1)}\cdot J_f^M(A_1).$$
Since $m(A_3)<(1+\epsilon)m(A_1),$ it follows that
$$0\leq J_f^M(A_3)-J_f^M(A_1)\leq \|\mu(f)\chi_{(0,\epsilon^2)}\|_{L_M}+\epsilon \cdot J_f^M(A_1)\leq \|\mu(f)\chi_{(0,\epsilon^2)}\|_{L_M}+\epsilon\|f\|_{L_M}.$$
	
Similarly, we have
$$0\leq J_f^M(A_3)-J_f^M(A_2)\leq \|\mu(f)\chi_{(0,\epsilon^2)}\|_{L_M}+\epsilon\|f\|_{L_M}.$$
Thus,
$$|J_f^M(A_1)-J_f^M(A_2)|\leq 2\|\mu(f)\chi_{(0,\epsilon^2)}\|_E+2\epsilon\|f\|_{L_M}\leq F_f(\epsilon^2),$$
where the final estimate above follows immediately from the definition of $F_f.$ This completes the proof in Case 1.
	
{\bf Case 2:} Let $m(A_1)\leq\epsilon$ or $m(A_2)\leq\epsilon.$ Since $m(A_1\bigtriangleup A_2)<\epsilon^2,$ it follows that we simultaneously have $m(A_1)\leq2\epsilon$ and $m(A_2)\leq2\epsilon.$ By the definition of $J_f^M$, we obtain
$$J_f^M(A_k)\leq 2\epsilon\|\sigma_{\frac1{2\epsilon}}\mu(f)\|_{L_M},\quad k=1,2.$$
Thus,
$$|J_f^M(A_1)-J_f^M(A_2)|\leq 4\epsilon\|\sigma_{\frac1{2\epsilon}}\mu(f)\|_{L_M}\leq F_f(\epsilon^2).$$
This completes the proof in Case 2.
\end{proof}

The following assertion is well known (see e.g. Theorem II.18.1 in \cite{dibenedetto} or Appendix B in \cite{FrankLaptev}). We recall that by a cube we always mean an open cube with edges parallel to the coordinate axes.

\begin{lem}[Besicovitch covering lemma]\label{besicovitch lemma} For every $x\in\mathbb{T}^d$ let $\Pi_x\subset \mathbb{T}^d$ be a closed cube centered in $x.$ There exists $c_d\in\mathbb{N}$ and subsets $(S_l)_{l=1}^{c_d}$ in $\mathbb{T}^d$ such that
\begin{enumerate}[{\rm (i)}]
\item $\mathbb{T}^d=\cup_{l=1}^{c_d}\cup_{x\in S_l}\Pi_x.$
\item $\Pi_{x_1}\cap\Pi_{x_2}=\varnothing$ for $x_1,x_2\in S_l,$ $x_1\neq x_2.$
\end{enumerate} 
\end{lem}

Proof of Theorem \ref{solomyak cover} follows the pattern established in \cite[p.260]{Solomyak1994} but covers the case of arbitrary dimension $d.$ According to G. Rozenblum, the idea to use coverings instead of partitions (as in earlier papers of Birman and Solomyak) belongs to him. In \cite{Solomyak1994} (see also earlier paper \cite{Birman-Solomyak-10th}), handcrafted covering lemma of Rozenblum was replaced by the Besicovitch covering lemma.

\begin{proof}[Proof of Theorem \ref{solomyak cover}] Fix $f\in L_M(\mathbb{T}^d)$. Let $\Pi_{x,t}$ be the closed cube centered in $x\in \mathbb{T}^d$ with a side $t\in (0,1).$ By Lemma \ref{j continuity}, the function
$$t\to J_f^M(\Pi_{x,t}),\quad t\in[0,1],$$
is continuous. By Intermediate Value Theorem, there exists $t=t(x)$ such that 
\begin{equation}\label{sc eq1}
J_f^M(\Pi_{x,t_x})=\frac1n\|f\|_{L_M}.
\end{equation}
	
Set $\Pi_x=\Pi_{x,t(x)},$ $x\in\mathbb{T}^d.$ Consider the covering $\{\Pi_x\}_{x\in \mathbb{T}^d}$ of $\mathbb{T}^d $. Let $c_d\in\mathbb{N}$ and sets $(S_l)_{l=1}^{c_d}$ be as in Lemma \ref{besicovitch lemma}. Consider an arbitrary finite subset $A_l\subset S_l$. Note that 
$$\Big\{\Pi_x\Big\}_{x\in A_l}\bigcup\Big\{\bigcap_{x\in A_l}\Pi_x^c\Big\}$$
is a partition of $\mathbb{T}^d.$ By \eqref{sc eq1} and Lemma \ref{lorentz lemma}, we have
$$|A_l|\cdot\frac1n\|f\|_{L_M}=\sum_{x\in A_l}J_f^M(\Pi_x)\leq J_f^M(\cap_{x\in A_l}\Pi_x^c)+\sum_{x\in A_l}J_f^M(\Pi_x)\leq 4\|f\|_E.$$
In other words, $|A_l|\leq 4n$ for every finite subset of $S_l.$ This implies that the set $S_l$ is finite and $|S_l|\leq 4n.$
	
Set $\Pi_k=\Pi_{l,x},$ where index $k$ stands for the couple $(l,x)$ with $x\in S_l.$ It follows from the preceding paragraph that there are at most $4c_dn$ distinct indices $k.$ This completes the proof.
\end{proof}

\section{Proof of Theorem \ref{solomyak cwikel estimate torus}}\label{solomyak thm torus}

The following fact is standard and is only supplied for convenience of the reader  and due to the lack of a proper reference. It asserts that the homogeneous semi-norm behaves well with respect to scaling.
\begin{fact}\label{scaling invariance lemma} Let $\Pi=[-\pi\epsilon,\pi\epsilon]^d,$ $0<\epsilon\leq 1.$ We have
$$\|\sigma_{\frac1{\epsilon}}u\|_{W^{s,2}_{{\rm hom}}([-\pi,\pi]^d)}=\epsilon^{s-\frac{d}{2}}\|u\|_{W^{s,2}_{{\rm hom}}(\Pi)},\quad u\in W^{s,2}(\Pi),\quad s>0.$$
In particular,
$$\|\sigma_{\frac1{\epsilon}}u\|_{W^{\frac{d}{2},2}_{{\rm hom}}([-\pi,\pi]^d)}=\|u\|_{W^{\frac{d}{2},2}_{{\rm hom}}(\Pi)},\quad u\in W^{\frac{d}{2},2}(\Pi).$$

\end{fact}

In the proof of next lemma, which is an extension of \cite[Lemma 2]{Solomyak1994} to the case of an arbitrary dimension, we crucially exploit the fact that the homogeneous norm behaves well with respect to scaling. 

\begin{lem}\label{scaled holder lemma} Let $d\in\mathbb{N}.$ Let $\Pi\subset\mathbb{T}^d$ be a cube. Let $M(t)=t\log(e+t),$ $t>0,$ and let $f\in L_M(\mathbb{T}^d).$  For every $u\in W^{\frac{d}{2},2}(\Pi)$ orthogonal (in $L_2(\Pi)$) with every polynomial of degree $<\frac{d}{2},$ we have
$$\int_{\Pi}|f|\cdot |u|^2\leq c_d J_f^M(\Pi)\cdot \|u\|_{W^{\frac{d}{2},2}_{{\rm hom}}(\Pi)}^2.$$
\end{lem}
\begin{proof} Without loss of generality, $\Pi=[-\pi\epsilon,\pi\epsilon]^d.$ 
	
By scaling, we have
$$\int_{\Pi}|f|\cdot |u|^2=\epsilon^d\int_{\mathbb{T}^d}|\sigma_{\frac1{\epsilon}}f|\cdot |\sigma_{\frac1{\epsilon}}u|^2.$$
By H\"older inequality (see e.g. Theorem II.5.2 in \cite{KPS}), we have
$$\int_{\mathbb{T}^d}F|G|^2\leq c_{{\rm abs}}\|F\|_{L_M(\mathbb{T}^d)}\||G|^2\|_{\exp(L_1)(\mathbb{T}^d)}= c_{{\rm abs}}\|F\|_{L_M(\mathbb{T}^d)}\|G\|_{\exp(L_2)(\mathbb{T}^d)}^2$$
for all $F\in L_M(\mathbb{T}^d)$ and for all $G\in\exp(L_2)(\mathbb{T}^d).$ Thus,
$$\int_{\Pi}|f|\cdot |u|^2\leq c_{{\rm abs}} \epsilon^d\|\sigma_{\frac1{\epsilon}}f\|_{L_M(\mathbb{T}^d)}\|\sigma_{\frac1{\epsilon}}u\|_{\exp(L_2)(\mathbb{T}^d)}^2.$$
Obviously, $\sigma_{\frac1{\epsilon}}u$ is orthogonal to every polynomial of degree $<\frac{d}{2}$ on $\mathbb{T}^d.$ By Theorem \ref{ozawa estimate} and Theorem \ref{comparison thm}, we have
$$\|\sigma_{\frac1{\epsilon}}u\|_{\exp(L_2)(\mathbb{T}^d)}\leq c_d\|\sigma_{\frac1{\epsilon}}u\|_{W^{\frac{d}{2},2}_{{\rm hom}}([-\pi,\pi]^d)}\stackrel{S.\ref{scaling invariance lemma}}{=}c_d\|u\|_{W^{\frac{d}{2},2}_{{\rm hom}}(\Pi)}.$$
By the definition of $J_f^M,$ we have
$$\epsilon^d\|\sigma_{\frac1{\epsilon}}f\|_{L_M(\mathbb{T}^d)}=J_f^M(\Pi).$$
A combination of the last three equations yields the assertion.
\end{proof}

The following fact is standard and is only supplied for convenience of the reader and due to the lack of a proper reference.

\begin{fact}\label{cube fact} Let $\Pi\subset\mathbb{T}^d$ be a cube and let $P:L_2(\Pi)\to L_2(\Pi)$ be the projection onto the subspace spanned by polynomials of degree $<\frac{d}{2}.$
\begin{enumerate}[{\rm (i)}]
\item\label{csa} for every $u\in L_2(\Pi),$ the function $u-Pu$ is orthogonal (in $L_2(\Pi)$) to every polynomial $v$ of degree $<\frac{d}{2}.$
\item\label{csb} for every $u\in W^{\frac{d}{2},2}(\Pi),$ we have $\|u-Pu\|_{W^{\frac{d}{2},2}_{{\rm hom}}(\Pi)}=\|u\|_{W^{\frac{d}{2},2}_{{\rm hom}}(\Pi)}.$ 
\end{enumerate}
\end{fact}
%

The following assertion was proved by Solomyak for even $d$ (see Theorem 1 in \cite{Solomyak1994}). We prove it for an arbitrary dimension.

\begin{lem}\label{solomyak intermediate} Let $d\in\mathbb{N}.$ Let $M(t)=t\log(e+t),$ $t>0,$ and let $0\leq f\in L_M(\mathbb{T}^d).$ For every $n\in\mathbb{N},$ there exists an operator $K_n:L_2(\mathbb{T}^d)\to L_2(\mathbb{T}^d)$ such that ${\rm rank}(K_n)\leq c_dn$ and such that
$$\int_{\mathbb{T}^d}f|u-K_nu|^2\leq\frac{c_d}n\big\|f\big\|_{L_M}\big\|u\big\|_{W^{\frac{d}{2},2}_{{\rm hom}}(\mathbb{T}^d)}^2,\quad u\in W^{\frac{d}{2},2}(\mathbb{T}^d).$$
We also have that $K_n:L_2(\mathbb{T}^d)\to L_{\infty}(\mathbb{T}^d).$
\end{lem}
\begin{proof} Let $(\Pi_k)_{1\leq k\leq m(n)}$ be the sequence of cubes constructed in Theorem \ref{solomyak cover}. 

Let $P_k:L_2(\mathbb{T}^d)\to L_2(\mathbb{T}^d)$ be the projection such that
$$P_k=M_{\chi_{\Pi_k}}P_kM_{\chi_{\Pi_k}},\quad 1\leq k\leq m(n)$$
and such that $P_k:L_2(\Pi_k)\to L_2(\Pi_k)$ is the projection onto the linear subspace of all polynomials of degree $<\frac{d}{2}.$
	
Set
$$\Delta_k=\Pi_k\backslash\bigcup_{l<k}\Pi_l,\quad 1\leq k\leq m(n).$$
By Theorem \ref{solomyak cover} \eqref{cova}, the sequence $(\Delta_k)_{k=1}^{m(n)}$ is a partition of $\mathbb{T}^d.$ Set
$$K_n=\sum_{k=1}^{m(n)}M_{\Delta_k}P_k.$$

From the definition, it is immediate that $K_n:L_2(\mathbb{T}^d)\to L_{\infty}(\mathbb{T}^d).$ Since $m(n)\leq c_dn$ by Theorem \ref{solomyak cover}, it follows that ${\rm rank}(K_n)\leq c_dn$ (with a different constant $c_d$). 

We have
$$\int_{\mathbb{T}^d}f|u-K_nu|^2=\sum_{k=1}^{m(n)}\int_{\Delta_k}f|u-K_nu|^2=\sum_{k=1}^{m(n)}\int_{\Delta_k}f|u-P_ku|^2.$$
Thus,
\begin{equation}\label{sil eq0}
\int_{\mathbb{T}^d}f|u-K_nu|^2\leq \sum_{k=1}^{m(n)}\int_{\Pi_k}f |u-P_ku|^2.
\end{equation}

By Scholium \ref{cube fact} \eqref{csa}, the function $u-P_ku$ satisfies the assumptions of Lemma \ref{scaled holder lemma}. By Lemma \ref{scaled holder lemma} and Scholium \ref{cube fact} \eqref{csb}, we have
$$\int_{\Pi_k}f|u-P_ku|^2\leq c_d J_f^M(\Pi_k)\cdot \|u-P_ku\|_{W^{\frac{d}{2},2}_{{\rm hom}}(\Pi_k)}^2=c_d J_f^M(\Pi_k)\cdot \|u\|_{W^{\frac{d}{2},2}_{{\rm hom}}(\Pi_k)}^2.$$
Combining the latter estimate with Theorem \ref{solomyak cover} \eqref{covc}, we obtain
$$\int_{\Pi_k}f|u-P_ku|^2\leq \frac{c_d}n\|f\|_{L_M}\cdot \|u\|_{W^{\frac{d}{2},2}_{{\rm hom}}(\Pi_k)}^2$$
and, therefore, by \eqref{sil eq0},
$$\int_{\mathbb{T}^d}f|u-K_nu|^2\leq\frac{c_d}n\|f\|_{L_M}\sum_{k=1}^{m(n)}\|u\|_{W^{\frac{d}{2},2}_{{\rm hom}}(\Pi_k)}^2.$$
Using Theorem \ref{solomyak cover} \eqref{covb} and Theorem \ref{sobolev torus vs sobolev cube}, we obtain
$$\sum_{k=1}^{m(n)}\|u\|_{W^{\frac{d}{2},2}_{{\rm hom}}(\Pi_k)}^2\leq c_d\|u\|_{W^{\frac{d}{2},2}_{{\rm hom}}([-\pi,\pi])}^2\leq c_d\|u\|_{W^{\frac{d}{2},2}_{{\rm hom}}(\mathbb{T}^d)}^2.$$
Combining the last two inequalities, we complete the proof.
\end{proof}

Approximation given in Lemma \ref{solomyak intermediate} above yields the quasi-norm estimate in a standard fashion (see schematic exposition on p.58 in \cite{Solomyak1995} and some earlier results e.g. Theorem 3.3 in \cite{BS67}).

\begin{proof}[Proof of Theorem \ref{solomyak cwikel estimate torus}] Without loss of generality, $f\geq0.$

Let $c_d$ be the constant in Lemma \ref{solomyak intermediate} (we assume this constant to be an integer). Take $m\in\mathbb{N}$ such that $m\geq 3c_d.$ Let $n\in\mathbb{N}$ be such that $m\in[3c_dn,3c_d(n+1)).$
	
Let the operator $K_n:L_2(\mathbb{T}^d)\to L_2(\mathbb{T}^d)$ be the one whose existence in Lemma \ref{solomyak intermediate}. We have that ${\rm rank}(K_n)\leq c_dn$ and that
$$\int_{\mathbb{T}^d}f|u-K_nu|^2\leq\frac{c_d}n\big\|f\big\|_{L_M}\big\|u\big\|_{W^{\frac{d}{2},2}}^2,\quad u\in W^{\frac{d}{2},2}(\mathbb{T}^d).$$
It is immediate that
\begin{align*}
\int_{\mathbb{T}^d}f|u-K_nu|^2&=\langle f\cdot u,u\rangle-\langle f\cdot u,K_nu\rangle-\langle f\cdot K_nu,u\rangle+\langle f\cdot K_nu,K_nu\rangle\\ &
=\langle M_fu,u\rangle-\langle K_n^{\ast}M_fu,u\rangle-\langle M_fK_nu,u\rangle+\langle K_n^{\ast}M_fK_nu,u\rangle\\ &=\langle T_nu,u\rangle,
\end{align*}
where
$$T_n=M_f-K_n^{\ast}M_f-M_fK_n+K_n^{\ast}M_fK_n.$$
Let us explain why the inner products in the equalities above exist. Note that $u\in W^{\frac{d}{2},2}(\mathbb{T}^d)\subset\exp(L_2)(\mathbb{T}^d)$ and that $K_nu\in L_{\infty}(\mathbb{T}^d)\subset \exp(L_2)(\mathbb{T}^d).$ It follows from H\"older inequality that
$$\|f_1f_2f_3\|_1\leq c_{{\rm abs}}\|f_1\|_{L_M}\|f_2f_3\|_{\exp(L_1)}\leq c_{{\rm abs}}\|f_1\|_{L_M}\|f_2\|_{\exp(L_2)}\|f_3\|_{\exp(L_2)}$$
whenever $f_1\in L_M(\mathbb{T}^d)$ and $f_2,f_3\in \exp(L_2)(\mathbb{T}^d).$

Thus,
$$|\langle T_nu,u\rangle|\leq \frac{c_d}n\big\|f\big\|_{L_M}\big\|(1-\Delta_{\mathbb{T}^d})^{\frac{d}{4}}u\big\|_2^2,\quad u\in W^{\frac{d}{2},2}(\mathbb{T}^d).$$
By definition, $(1-\Delta_{\mathbb{T}^d})^{\frac{d}{4}}$ is a bijection from  $W^{\frac{d}{2},2}(\mathbb{T}^d)$ to $L_2(\mathbb{T}^d).$ We, therefore, have
$$|\langle T_n(1-\Delta_{\mathbb{T}^d})^{-\frac{d}{4}}v,(1-\Delta_{\mathbb{T}^d})^{-\frac{d}{4}}v\rangle|\leq \frac{c_d}n\big\|f\big\|_{L_M}\big\|v\big\|_2^2,\quad v\in L_2(\mathbb{T}^d).$$
Thus,
$$|\langle (1-\Delta_{\mathbb{T}^d})^{-\frac{d}{4}}T_n(1-\Delta_{\mathbb{T}^d})^{-\frac{d}{4}}v,v\rangle|\leq \frac{c_d}n\big\|f\big\|_{L_M}\big\|v\big\|_2^2,\quad v\in L_2(\mathbb{T}^d).$$
Since $T_n$ is self-adjoint, we infer from the definition of the operator norm that
$$\Big\|(1-\Delta_{\mathbb{T}^d})^{-\frac{d}{4}}T_n(1-\Delta_{\mathbb{T}^d})^{-\frac{d}{4}}\Big\|_{\infty}\leq \frac{c_d}n\big\|f\big\|_{L_M}.$$
Using the notation
$$S_n=(1-\Delta_{\mathbb{T}^d})^{-\frac{d}{4}}\cdot\big(K_n^{\ast}M_f+M_fK_n-K_n^{\ast}M_fK_n\big)\cdot (1-\Delta_{\mathbb{T}^d})^{-\frac{d}{4}},$$
we rewrite the above inequality as
$$\Big\|(1-\Delta_{\mathbb{T}^d})^{-\frac{d}{4}}M_f(1-\Delta_{\mathbb{T}^d})^{-\frac{d}{4}}-S_n\Big\|_{\infty}\leq \frac{c_d}n\big\|f\big\|_{L_M}.$$
Since the rank of operator $K_n$ (and, hence, of the operator $K_n^{\ast}$) does not exceed $c_dn,$ it follows that ${\rm rank}(S_n)\leq 3c_dn.$ Hence,
$$\inf_{{\rm rank}(S)\leq 3c_dn}\Big\|(1-\Delta_{\mathbb{T}^d})^{-\frac{d}{4}}M_f(1-\Delta_{\mathbb{T}^d})^{-\frac{d}{4}}-S\Big\|_{\infty}\leq \frac{c_d}n\big\|f\big\|_{L_M}.$$
That is
$$\mu\big(3c_dn,(1-\Delta_{\mathbb{T}^d})^{-\frac{d}{4}}M_f(1-\Delta_{\mathbb{T}^d})^{-\frac{d}{4}}\big)\leq \frac{c_d}n\big\|f\big\|_{L_M}.$$
Since 
$$\frac{c_d}{n}\leq\frac{6c_d^2}{n+1},$$
it follows that
\begin{equation}\label{weak est eq1}
\mu\big(m,(1-\Delta_{\mathbb{T}^d})^{-\frac{d}{4}}M_f(1-\Delta_{\mathbb{T}^d})^{-\frac{d}{4}}\big)\leq \frac{6c_d^2}{m+1}\big\|f\big\|_{L_M},\quad m\geq 3c_d.
\end{equation}

Now, for $m\in\mathbb{Z}_+$ with $m<3c_d,$ we have
\begin{align*}
\mu\big(m,(1-\Delta_{\mathbb{T}^d})^{-\frac{d}{4}}M_f(1-\Delta_{\mathbb{T}^d})^{-\frac{d}{4}}\big)&\leq \mu\big(0,(1-\Delta_{\mathbb{T}^d})^{-\frac{d}{4}}M_f(1-\Delta_{\mathbb{T}^d})^{-\frac{d}{4}}\big)\\
&=\Big\|(1-\Delta_{\mathbb{T}^d})^{-\frac{d}{4}}M_f(1-\Delta_{\mathbb{T}^d})^{-\frac{d}{4}}\Big\|_{\infty}\\
& \leq c_d\|f\|_{L_M}\leq\frac{3c_d^2}{m+1}\|f\|_{L_M}.
\end{align*}
Hence, \eqref{weak est eq1} also holds for $m<3c_d.$ Thus,
$$\Big\|(1-\Delta_{\mathbb{T}^d})^{-\frac{d}{4}}M_f(1-\Delta_{\mathbb{T}^d})^{-\frac{d}{4}}\Big\|_{1,\infty}\leq 6c_d^2\|f\|_{L_M}.$$
\end{proof}

\section{Symmetrized Cwikel-type estimate for $\mathcal{L}_{1,\infty}$ in $\mathbb{R}^d$}	

This section is devoted to the proof of Theorem \ref{best available rd cwikel}.

\subsection{The function $f$ is supported on the unit cube}

When $f$ is supported on $(-1,1)^d,$ we may extend $f$ to a function on $\mathbb{T}^d$ (e.g., by identifying $\mathbb{T}^d$ with $[-\pi,\pi]^d$ and by setting $f=0$ on $[-\pi,\pi]^2\backslash[-1,1]^d$).

\begin{lem}\label{dao compactification lemma} Let $0\leq f\in L_{\infty}(\mathbb{R}^d)$ be supported on $(-1,1)^d.$ We have
$$M_{f^{\frac12}}(1-\Delta_{\mathbb{R}^d})^{-\frac{d}{2}}M_{f^{\frac12}}\Big|_{L_2((-1,1)^d)}=M_{f^{\frac12}}a(\nabla_{\mathbb{T}^d})M_{f^{\frac12}}\Big|_{L_2((-1,1)^d)},$$
where $a\in l_{\infty}(\mathbb{Z}^d)$ does not depend on $f$ and is such that
$$|a(n)|\leq c_d(1+|n|^2)^{-\frac{d}{2}},\quad n\in\mathbb{Z}^d.$$
\end{lem}
\begin{proof} This is, effectively, a combination of Lemmas 4.5 and 4.6 in \cite{SZ-DAO1}. There, one deals with the cube $(0,1)^d,$ but taking $(-1,1)^d$ instead makes no difference.
\end{proof}

The following lemma yields the assertion of Theorem \ref{best available rd cwikel} in the special case when $f$ is supported on the cube $(-1,1)^d.$ Recall that $M(t)=t\log(e+t),$ $t>0.$

\begin{lem}\label{f inside ball lemma} Let $f\in L_{\infty}(\mathbb{R}^d)$ be supported on $(-1,1)^d.$ We have
$$\Big\|(1-\Delta_{\mathbb{R}^d})^{-\frac{d}{4}}M_f(1-\Delta_{\mathbb{R}^d})^{-\frac{d}{4}}\Big\|_{1,\infty}\leq c_d\|f\chi_{(-1,1)^d}\|_{L_M}.$$ 
\end{lem}
\begin{proof} Without loss of generality, $f\geq0.$ As established in \cite{RSZ}, the operator
$$(1-\Delta_{\mathbb{R}^d})^{-\frac{d}{4}}M_f(1-\Delta_{\mathbb{R}^d})^{-\frac{d}{4}}$$
is bounded. Using the standard identities
$$\mu(TT^{\ast})=\mu(T^{\ast}T)\mbox{ and }\|TT^{\ast}\|_{1,\infty}=\|T^{\ast}T\|_{1,\infty},$$
we conclude that
$$\Big\|(1-\Delta_{\mathbb{R}^d})^{-\frac{d}{4}}M_f(1-\Delta_{\mathbb{R}^d})^{-\frac{d}{4}}\Big\|_{1,\infty}=\Big\|M_{f^{\frac12}}(1-\Delta_{\mathbb{R}^d})^{-\frac{d}{2}}M_{f^{\frac12}}\Big\|_{1,\infty}.$$
By Lemma \ref{dao compactification lemma}, we have
$$\|M_{f^{\frac12}}(1-\Delta_{\mathbb{R}^d})^{-\frac{d}{2}}M_{f^{\frac12}}\Big\|_{1,\infty}=\|M_{f^{\frac12}}a(\nabla_{\mathbb{T}^d})M_{f^{\frac12}}\|_{1,\infty}\leq $$
$$\leq c_d\|M_{f^{\frac12}}(1-\Delta_{\mathbb{T}^d})^{-\frac{d}{2}}M_{f^{\frac12}}\|_{1,\infty}=c_d\Big\|(1-\Delta_{\mathbb{T}^d})^{-\frac{d}{4}}M_f(1-\Delta_{\mathbb{T}^d})^{-\frac{d}{4}}\Big\|_{1,\infty}.$$
The assertion follows now from Theorem \ref{solomyak cwikel estimate torus}.
\end{proof}

\subsection{The function $f$ is supported outside of the unit ball}

In what follows, we equip the unit ball $\mathbb{B}^d$ in $\mathbb{R}^d$ with Lebesgue measure.

Let $V:L_1(\mathbb{R}^d)\to L_1(\mathbb{R}^d)$ be an isometry given by the formula
$$(Vf)(t)=|t|^{-2d}f(\frac{t}{|t|^2}),\quad f\in L_1(\mathbb{R}^d).$$
That $V$ is indeed an isometry is proved below in Lemma \ref{urd def lemma}.

\begin{lem}\label{urd def lemma} The operator 
$$(U\xi)(t)=|t|^{-d}\cdot \xi\big(\frac{t}{|t|^2}\big),\quad \xi\in L_2(\mathbb{R}^d),$$
is unitary on $L_2(\mathbb{R}^d).$
\end{lem}
\begin{proof} Let $s_k=\frac{t_k}{|t|^2}.$ We have
$$\frac{\partial s_k}{\partial t_l}=-\frac{2t_kt_l}{|t|^4},\quad k\neq l,$$ 
$$\frac{\partial s_k}{\partial t_k}=\frac{|t|^2-2t_k^2}{|t|^4}.$$
Hence, one can write the Jacobian as
$$J=|t|^{-2}\cdot\Big(1-2\Big(\frac{t_k}{|t|}\cdot\frac{t_l}{|t|}\Big)_{1\leq k,l\leq d}\Big).$$
Obviously, the matrix
$$\Big(\frac{t_k}{|t|}\cdot\frac{t_l}{|t|}\Big)_{1\leq k,l\leq d}$$
is rank $1$ projection on the Hilbert space $\mathbb{C}^d.$ In other words, it is unitarily equivalent to a matrix unit $E_{11}$ (that is, to the matrix whose $(1,1)-$entry is $1$ and whose other entries are zeroes). Hence, 
$${\rm det}(J)=|t|^{-2d}\cdot {\rm det}(1-2E_{11})=-|t|^{-2d}.$$
	
It follows that
$$\int_{\mathbb{R}^d}\eta(s)ds=\int_{\mathbb{R}^d}\eta(\frac{t}{|t|^2})\cdot |{\rm det}(J)(t)|dt=\int_{\mathbb{R}^d}\eta(\frac{t}{|t|^2})\cdot |t|^{-2d}dt.$$
Setting $\eta=|\xi|^2,$ we can write
$$\int_{\mathbb{R}^d}|\xi|^2(s)ds=\int_{\mathbb{R}^d}|\xi|^2(\frac{t}{|t|^2})\cdot |t|^{-2d}dt.$$
In other words,
$$\|\xi\|_{L_2(\mathbb{R}^d)}^2=\|U\xi\|_{L_2(\mathbb{R}^d)}^2.$$
\end{proof}

It is important to note that $U=U^{-1}.$ The following lemma can be either established via (long) direct calculation or derived from general geometric results  (see e.g. from Section III.7 in \cite{Kobayashi}). The symbol $\partial_k$ denotes the partial derivative with respect to the $k$-th coordinate.

\begin{lem}\label{urd laplacian lemma} We have
$$U^{-1}\Delta_{\mathbb{R}^d} U=U\Delta_{\mathbb{R}^d} U^{-1}=\sum_{k=1}^dM_{h_d}\partial_kM_{h_{4-2d}}\partial_kM_{h_d}.$$
Here, $h_z(t)=|t|^z,$ $t\in\mathbb{R}^d.$
\end{lem}

\begin{cor}\label{urd laplacian corollary} We have
$$U(1-\Delta_{\mathbb{R}^d})^nU^{-1}=\sum_{|\gamma|_1\leq 2n}\partial^{\gamma}M_{p_{\gamma}},\quad {\rm deg}(p_{\gamma})\leq 4n.$$
Here, the polynomials $p_{\gamma}$ with $|\gamma|_1=2n$ are of order $4n$ (in fact, they are the scalar multiples of $h_{4n}$), while the polynomials $p_{\gamma}$ with $|\gamma|_1<2n$ have lower order. 
\end{cor}
\begin{proof} By Lemma \ref{urd laplacian lemma},
$$U(1-\Delta_{\mathbb{R}^d})U^{-1}=\Delta_{\mathbb{R}^d} M_{h_4}+c_d\sum_{k=1}^d\partial_k M_{\partial_kh_4}+c_d'M_{h_2}$$
is a differential operator of order $2$ with polynomial coefficients of degree $4$ or less. Hence, $U(1-\Delta_{\mathbb{R}^d})^nU^{-1}$ is a differential operator of order $2n$ with polynomial coefficients of degree $4n$ or less. The degree of polynomials $p_{\gamma}$ can be inferred from the Leibniz rule.
\end{proof}

\begin{fact}\label{sigma2 lemma} For all $S,T\in\mathcal{L}_{\infty},$ we have
$$\mu(TSS^{\ast}T^{\ast})\leq \|S\|_{\infty}^2\sigma_2\mu(TT^{\ast}).$$
\end{fact}
\begin{proof} Indeed,
$$\mu(TSS^{\ast}T^{\ast})=\mu^2(ST)\leq\|S\|_{\infty}^2\mu^2(T)=\|S\|_{\infty}^2\mu(TT^{\ast}).$$
\end{proof}

Let $C^n(\mathbb{R}^d)$ be the collection of all $n$ times continuously differentiable functions such that the function itself and all its derivatives up to order $n$ are bounded.

\begin{fact}\label{standard boundedness fact} Suppose $g\in C^{2n}(\mathbb{R}^d).$ We have
$$\|\partial^{\gamma}M_g(1-\Delta_{\mathbb{R}^d})^{-n}\|_{\infty}\leq c_{n,\gamma}\|g\|_{C^{2n}(\mathbb{R}^d)},\quad |\gamma|_1\leq 2n.$$
\end{fact}
\begin{proof} We have
$$\partial^{\gamma}M_g=\sum_{\substack{\gamma_1+\gamma_2=\gamma\\ \gamma_1,\gamma_2\geq0}}c_{\gamma_1,\gamma_2}M_{\partial^{\gamma_1}g}\partial^{\gamma_2}.$$
Therefore,
$$\|\partial^{\gamma}M_g(1-\Delta_{\mathbb{R}^d})^{-n}\|_{\infty}\leq\sum_{\substack{\gamma_1+\gamma_2=\gamma\\ \gamma_1,\gamma_2\geq0}}|c_{\gamma_1,\gamma_2}|\|M_{\partial^{\gamma_1}g}\|_{\infty}\|\partial^{\gamma_2}(1-\Delta_{\mathbb{R}^d})^{-n}\|_{\infty}.$$
The operator $\partial^{\gamma_2}(1-\Delta_{\mathbb{R}^d})^{-n}$ on the right hand side is bounded by the functional calculus. By assumption, we have
$$\|M_{\partial^{\gamma_1}g}\|_{\infty}\leq\|g\|_{C^{2n}(\mathbb{R}^d)}$$
and the assertion follows.
\end{proof}

The following lemma (for $z=\frac{d}{4}$) is the crucial technical tool in the proof of Theorem \ref{best available rd cwikel}. Its proof relies on Hadamard three lines theorem.

\begin{lem}\label{main boundedness lemma} For every real-valued $\phi\in C^{\infty}_c(\mathbb{R}^d),$ the operator
$$T_z=(1-\Delta_{\mathbb{R}^d})^zM_{h_{4z}\phi}U^{-1}(1-\Delta_{\mathbb{R}^d})^{-z},\quad z\in\mathbb{C},\quad\Re(z)\geq0,$$
is well defined and bounded on $L_2(\mathbb{R}^d).$ Here, $h_z(t)=|t|^z,$ $t\in\mathbb{R}^d.$
\end{lem}
\begin{proof} First, note that the operator $M_{h_{4z}\phi}U^{-1}(1-\Delta_{\mathbb{R}^d})^{-z}$ is bounded on $L_2(\mathbb{R}^d)$ (as a composition of bounded operators). If $\xi\in L_2(\mathbb{R}^d),$ then $M_{h_{4z}\phi}U^{-1}(1-\Delta_{\mathbb{R}^d})^{-z}\xi$ is also an element of $L_2(\mathbb{R}^d)$ and is, therefore, a tempered distribution. Hence, $T_z\xi=(1-\Delta_{\mathbb{R}^d})^zM_{h_{4z}\phi}U^{-1}(1-\Delta_{\mathbb{R}^d})^{-z}\xi$ is also a tempered distribution. We aim to show that the latter tempered distribution is actually an element of $L_2(\mathbb{R}^d).$
	
Let $\eta\in\mathcal{S}(\mathbb{R}^d)$ (i.e. $\eta$ is a Schwartz function). Consider the function
$$F:z\to\langle T_z\xi,\eta\rangle=\langle M_{h_{4z}\phi}U^{-1}(1-\Delta_{\mathbb{R}^d})^{-z}\xi,(1-\Delta_{\mathbb{R}^d})^{\bar{z}}\eta\rangle,\quad \Re(z)\geq0.$$
The function
$$z\to M_{h_{4z}\phi}U^{-1}(1-\Delta_{\mathbb{R}^d})^{-z}\xi,\quad \Re(z)\geq0,$$
is $L_2(\mathbb{R}^d)-$valued analytic (and continuous on the boundary). The function
$$z\to (1-\Delta_{\mathbb{R}^d})^{\bar{z}}\eta,\quad\Re(z)\geq0,$$
is $L_2(\mathbb{R}^d)-$valued anti-analytic (and continuous on the boundary). Thus, $F$ is analytic and continuous on the boundary.
	
We have
\begin{align*}
|F(i\lambda)|&\leq\|M_{h_{4i\lambda}\phi}U^{-1}(1-\Delta_{\mathbb{R}^d})^{-i\lambda}\xi\|_{L_2(\mathbb{R}^d)}\|(1-\Delta_{\mathbb{R}^d})^{-i\lambda}\eta\|_{L_2(\mathbb{R}^d)}\\&
\leq \|\phi\|_{L_{\infty}(\mathbb{R}^d)}\|\xi\|_{L_2(\mathbb{R}^d)}\|\eta\|_{L_2(\mathbb{R}^d)}. 
\end{align*}

Also,
\begin{align*}
|F(n+i\lambda)|&\leq\|T_{n+i\lambda}\xi\|_{L_2(\mathbb{R}^d)}\|\eta\|_{L_2(\mathbb{R}^d)}\\&
\leq
\|U(1-\Delta_{\mathbb{R}^d})^nM_{h_{4n+4i\lambda}\phi}U^{-1}(1-\Delta_{\mathbb{R}^d})^{-n}\|_{\infty}\|\xi\|_{L_2(\mathbb{R}^d)}\|\eta\|_{L_2(\mathbb{R}^d)}.
\end{align*}

Denote, for brevity, $\alpha(t)=\frac{t}{|t|^2},$ $t\in\mathbb{R}^d.$ By Corollary \ref{urd laplacian corollary}, we have
\begin{align*}
U(1-\Delta_{\mathbb{R}^d})^n&M_{h_{4n+4i\lambda}\phi}U^{-1}(1-\Delta_{\mathbb{R}^d})^{-n}\\ &=U(1-\Delta_{\mathbb{R}^d})^nU^{-1}\cdot UM_{h_{4n+4i\lambda}\phi}U^{-1}\cdot (1-\Delta_{\mathbb{R}^d})^{-n}\\ &
=\sum_{|\gamma|_1\leq 2n}\partial^{\gamma}M_{p_{\gamma}}\cdot M_{h_{-4n-4i\lambda}\cdot (\phi\circ\alpha)}\cdot (1-\Delta_{\mathbb{R}^d})^{-n},
\end{align*}
where the last equality follows from 
$$UM_{h_z\phi}U^{-1}=M_{h_{-z}\cdot (\phi\circ\alpha)},\quad z\in\mathbb{C}.$$

Note that $\phi\circ\alpha$ vanishes near $0.$ Fix $\epsilon>0$ such that $\phi\circ\alpha=0$ on $\epsilon\mathbb{B}^d.$ An elementary calculation shows that
$$p_{\gamma}\cdot h_{-4n-4i\lambda}\in C^{2n}(\mathbb{R}^d\backslash\epsilon\mathbb{B}^d)$$
and, moreover,
$$\|p_{\gamma}\cdot h_{-4n-4i\lambda}\|_{C^{2n}(\mathbb{R}^d\backslash\epsilon\mathbb{B}^d)}\leq c_{n,\gamma}(1+|\lambda|)^{2n}.$$
Therefore,
$$p_{\gamma}\cdot h_{-4n-4i\lambda}\cdot (\phi\circ\alpha)\in C^{2n}(\mathbb{R}^d),$$
and
$$\|p_{\gamma}\cdot h_{-4n-4i\lambda}\cdot (\phi\circ\alpha)\|_{C^{2n}(\mathbb{R}^d)}\leq c_{n,\gamma,\phi}(1+|\lambda|)^{2n}.$$
By triangle inequality and Scholium \ref{standard boundedness fact} we have
\begin{align*}
\|U(1-\Delta_{\mathbb{R}^d})^nM_{h_{4n+4i\lambda}\phi}U^{-1}(1-\Delta_{\mathbb{R}^d})^{-n}\|_{\infty}&\leq \sum_{|\gamma|_1\leq 2n}c_{n,\gamma}c_{n,\gamma,\phi}(1+|\lambda|)^{2n}\\&
=c_{n,\phi}(1+|\lambda|)^{2n}.
\end{align*}

We conclude that
$$|F(n+i\lambda)|\leq c_{n,\phi}(1+|\lambda|)^{2n}\|\xi\|_{L_2(\mathbb{R}^d)}\|\eta\|_{L_2(\mathbb{R}^d)}.$$

Next, we claim that $F$ is bounded on the strip $\{0\leq \Re(z)\leq n\}.$ Indeed,
\begin{align*}
|F(z)|&\leq\|h_{4z}\cdot\phi\|_{L_{\infty}(\mathbb{R}^d)}\|(1-\Delta_{\mathbb{R}^d})^{-z}\xi\|_{L_2(\mathbb{R}^d)}\|(1-\Delta_{\mathbb{R}^d})^{\bar{z}}\eta\|_{L_2(\mathbb{R}^d)}\\ &
\leq c_{n,\phi}'\|\xi\|_{L_2(\mathbb{R}^d)}\|(1-\Delta_{\mathbb{R}^d})^n\eta\|_{L_2(\mathbb{R}^d)}.
\end{align*}

Let
$$G(z)=e^{z^2}F(z),\quad \Re(z)\geq0.$$
It follows that
$$|G(i\lambda)|\leq \|\phi\|_{L_{\infty}(\mathbb{R}^d)}\|\xi\|_{L_2(\mathbb{R}^d)}\|\eta\|_{L_2(\mathbb{R}^d)},$$
$$|G(n+i\lambda)|\leq c''_{n,\phi}\|\xi\|_{L_2(\mathbb{R}^d)}\|\eta\|_{L_2(\mathbb{R}^d)}.$$
In addition to that, the function $G$ is bounded on the strip $\{0\leq \Re(z)\leq n\}$ as the function $F$ is bounded there. We are now in a position to apply Hadamard three lines theorem, which yields
$$|G(z)|\leq \max\{\|\phi\|_{L_{\infty}(\mathbb{R}^d)},c'_{n,\phi}\}\cdot \|\xi\|_{L_2(\mathbb{R}^d)}\|\eta\|_{L_2(\mathbb{R}^d)},\quad 0\leq \Re(z)\leq n.$$
Therefore,
$$|F(z)|\leq |e^{-z^2}|\cdot \max\{\|\phi\|_{L_{\infty}(\mathbb{R}^d)},c'_{n,\phi}\}\cdot \|\xi\|_{L_2(\mathbb{R}^d)}\|\eta\|_{L_2(\mathbb{R}^d)},\quad 0\leq \Re(z)\leq n.$$
In other words, the functional
$$\eta\to \langle T_z\xi,\eta\rangle,\quad \eta\in\mathcal{S}(\mathbb{R}^d),$$
extends to a bounded functional on $L_2(\mathbb{R}^d)$ (and the norm of this functional is controlled by $c_z\|\xi\|_{L_2(\mathbb{R}^d)}$). By Riesz lemma, we have $T_z\xi\in L_2(\mathbb{R}^d)$ and
$$\|T_z\xi\|_{L_2(\mathbb{R}^d)}\leq c_z\|\xi\|_{L_2(\mathbb{R}^d)}.$$
Since $\xi\in L_2(\mathbb{R}^d)$ is arbitrary, it follows that $T_z$ is well defined and bounded on $L_2(\mathbb{R}^d).$ 
\end{proof}

The assertion of Lemma \ref{mu inversion lemma} is of crucial importance in the proof of Theorem \ref{best available rd cwikel}.

\begin{lem}\label{mu inversion lemma} Suppose $f\in L_{\infty}(\mathbb{R}^d)$ is supported on the set $\mathbb{R}^d\backslash\mathbb{B}^d.$ We have
$$\mu\Big(M_f(1-\Delta_{\mathbb{R}^d})^{-\frac{d}{2}}M_f\Big)\leq c_{{\rm abs}}\mu\Big(M_{Uf}(1-\Delta_{\mathbb{R}^d})^{-\frac{d}{2}}M_{Uf}\Big).$$
\end{lem}
\begin{proof} Denote, for brevity, $\alpha(t)=\frac{t}{|t|^2},$ $t\in\mathbb{R}^d.$ We have
$$U^{-1}\cdot M_f(1-\Delta_{\mathbb{R}^d})^{-\frac{d}{2}}M_f\cdot U=M_{f\circ\alpha}\cdot U^{-1}(1-\Delta_{\mathbb{R}^d})^{-\frac{d}{2}}U\cdot M_{f\circ\alpha}.$$
	
Fix real-valued function $\phi\in C^{\infty}_c(\mathbb{R}^d)$ such that $\phi=1$ on $\mathbb{B}^d.$ Since $f\circ\alpha$ is supported on $\mathbb{B}^d,$ it follows that
$$f\circ\alpha = (f\circ\alpha)\cdot\phi=Uf\cdot h_d\phi.$$
Thus,
\begin{align*}
M_{f\circ\alpha}&\cdot U^{-1}(1-\Delta_{\mathbb{R}^d})^{-\frac{d}{2}}U\cdot M_{f\circ\alpha}\\& 
=M_{Uf}\cdot M_{h_d\phi} U^{-1}(1-\Delta_{\mathbb{R}^d})^{-\frac{d}{2}}UM_{h_d\phi}\cdot M_{Uf}=TSS^{\ast}T^{\ast},
\end{align*}
where
$$T=M_{Uf}(1-\Delta_{\mathbb{R}^d})^{-\frac{d}{4}},\quad S=(1-\Delta_{\mathbb{R}^d})^{\frac{d}{4}}M_{h_d\phi}U^{-1}(1-\Delta_{\mathbb{R}^d})^{-\frac{d}{4}}.$$
Combining Lemma \ref{main boundedness lemma} and Scholium \ref{sigma2 lemma}, we complete the proof.
\end{proof}

\subsection{Proof of Theorem \ref{best available rd cwikel}}

The proof of the following proposition is postponed to the Appendix \ref{app section}.

\begin{prop}\label{equivalence prop} We have
$$\|f\chi_{\mathbb{B}^d}\|_{L_M(\mathbb{R}^d)}+\|(Vf)\chi_{\mathbb{B}^d}\|_{L_M(\mathbb{R}^d)}\approx \|f\|_{L_M(\mathbb{R}^d)}+\int_{\mathbb{R}^d}|f(s)|\log(1+|s|)ds.$$
\end{prop}

We are now ready to prove the main result in this section.

\begin{proof}[Proof of Theorem \ref{best available rd cwikel}] Without loss of generality, $f\geq0.$ Assume firstly $f\in L_{\infty}(\mathbb{R}^d).$ 

Obviously,
\begin{align*}
(1-\Delta_{\mathbb{R}^d})^{-\frac{d}{4}}M_f(1-\Delta_{\mathbb{R}^d})^{-\frac{d}{4}}&=(1-\Delta_{\mathbb{R}^d})^{-\frac{d}{4}}M_{f\chi_{\mathbb{B}^d}}(1-\Delta_{\mathbb{R}^d})^{-\frac{d}{4}}\\& 
+(1-\Delta_{\mathbb{R}^d})^{-\frac{d}{4}}M_{f\chi_{\mathbb{R}^d\backslash\mathbb{B}^d}}(1-\Delta_{\mathbb{R}^d})^{-\frac{d}{4}}.
\end{align*}
By the quasi-triangle inequality, we have
\begin{align*}
\Big\|(1-\Delta_{\mathbb{R}^d})^{-\frac{d}{4}}M_f(1-\Delta_{\mathbb{R}^d})^{-\frac{d}{4}}\Big\|_{1,\infty}&\leq 
2\Big\|(1-\Delta_{\mathbb{R}^d})^{-\frac{d}{4}}M_{f\chi_{\mathbb{B}^d}}(1-\Delta_{\mathbb{R}^d})^{-\frac{d}{4}}\Big\|_{1,\infty}\\&
+2\Big\|(1-\Delta_{\mathbb{R}^d})^{-\frac{d}{4}}M_{f\chi_{\mathbb{R}^d\backslash\mathbb{B}^d}}(1-\Delta_{\mathbb{R}^d})^{-\frac{d}{4}}\Big\|_{1,\infty}.
\end{align*}
By Lemma \ref{mu inversion lemma} applied to the function $f^{\frac12}\chi_{\mathbb{R}^d\backslash\mathbb{B}^d},$ we have
\begin{align*}
\Big\|(1-\Delta_{\mathbb{R}^d})^{-\frac{d}{4}}&M_{f\chi_{\mathbb{R}^d\backslash\mathbb{B}^d}}(1-\Delta_{\mathbb{R}^d})^{-\frac{d}{4}}\Big\|_{1,\infty}=\Big\|M_{f^{\frac12}\chi_{\mathbb{R}^d\backslash\mathbb{B}^d}}(1-\Delta_{\mathbb{R}^d})^{-\frac{d}{2}}M_{f^{\frac12}\chi_{\mathbb{R}^d\backslash\mathbb{B}^d}}\Big\|_{1,\infty}\\ &
\stackrel{L.\ref{mu inversion lemma}}{\leq} c_{{\rm abs}}\Big\|M_{U(f^{\frac12}\chi_{\mathbb{R}^d\backslash\mathbb{B}^d})}(1-\Delta_{\mathbb{R}^d})^{-\frac{d}{2}}M_{U(f^{\frac12}\chi_{\mathbb{R}^d\backslash\mathbb{B}^d})}\Big\|_{1,\infty}\\&
\leq c_{{\rm abs}}\Big\|(1-\Delta_{\mathbb{R}^d})^{-\frac{d}{4}}M_{(U(f^{\frac12}\chi_{\mathbb{R}^d\backslash\mathbb{B}^d}))^2}(1-\Delta_{\mathbb{R}^d})^{-\frac{d}{4}}\Big\|_{1,\infty}\\&
\leq c_{{\rm abs}}\Big\|(1-\Delta_{\mathbb{R}^d})^{-\frac{d}{4}}M_{(Vf) \chi_{\mathbb{B}^d}}(1-\Delta_{\mathbb{R}^d})^{-\frac{d}{4}}\Big\|_{1,\infty}.
\end{align*}
By Lemma \ref{f inside ball lemma}, we have
$$\Big\|(1-\Delta_{\mathbb{R}^d})^{-\frac{d}{4}}M_f(1-\Delta_{\mathbb{R}^d})^{-\frac{d}{4}}\Big\|_{1,\infty}\leq c_d\Big(\|f\chi_{\mathbb{B}^d}\|_{L_M(\mathbb{R}^d)}+\|(Vf)\chi_{\mathbb{B}^d}\|_{L_M(\mathbb{R}^d)}\Big).$$
The assertion (for bounded $f$) follows now from Proposition \ref{equivalence prop}.

Now, let $f\in L_M(\mathbb{R}^d)$ be arbitrary. Set
$$f_n=f\chi_{\{|f|\leq n\}},\quad n\in\mathbb{N}.$$
We already established the inequality for bounded function (in particular, the inequality holds for $f_n$). For every $n\in\mathbb{N},$ we have
$$\Big\|(1-\Delta_{\mathbb{R}^d})^{-\frac{d}{4}}M_{f_n}(1-\Delta_{\mathbb{R}^d})^{-\frac{d}{4}}\Big\|_{1,\infty}\leq c_d\Big(\|f\|_{L_M(\mathbb{R}^d)}+\int_{\mathbb{R}^d}|f(s)|\log(1+|s|)ds\Big).$$
On the other hand, it follows from Theorem 2.3 in \cite{LSZ-last-kalton} (the Lorentz space $\Lambda_1(\mathbb{R}^d)$ in \cite{LSZ-last-kalton} is known to coincide with the space $L_M(\mathbb{R}^d)$) that
$$\Big\|(1-\Delta_{\mathbb{R}^d})^{-\frac{d}{4}}M_{f_n}(1-\Delta_{\mathbb{R}^d})^{-\frac{d}{4}}-(1-\Delta_{\mathbb{R}^d})^{-\frac{d}{4}}M_f(1-\Delta_{\mathbb{R}^d})^{-\frac{d}{4}}\Big\|_{\infty}\leq c_d\|f-f_n\|_{L_M(\mathbb{R}^d)}.$$
It is easy to see that
$$\|f-f_n\|_{L_M(\mathbb{R}^d)}\to0,\quad n\to\infty.$$
It follows from the Fatou property of $\mathcal{L}_{1,\infty}$ that
$$\Big\|(1-\Delta_{\mathbb{R}^d})^{-\frac{d}{4}}M_f(1-\Delta_{\mathbb{R}^d})^{-\frac{d}{4}}\Big\|_{1,\infty}\leq c_d\Big(\|f\|_{L_M(\mathbb{R}^d)}+\int_{\mathbb{R}^d}|f(s)|\log(1+|s|)ds\Big).$$
\end{proof}

\section{Symmetrized Cwikel estimate for $\mathcal{L}_{1,\infty}$ does not hold in $\mathbb{R}^d$}

This section is devoted to the proof of Theorem \ref{bad rd theorem}.

\subsection{Simple facts used in the proof}

In the following lemma, the notation $\oplus_{k\in\mathbb{Z}^d}T_k$ is a shorthand for an element $\sum_{k\in\mathbb{Z}^d}T_k\otimes e_k$ in the von Neumann algebra $B(H)\bar{\otimes}l_{\infty}(\mathbb{Z}^d).$ Here, $e_k$ is the unit vector having the only non-zero component on the $k$-th position.

Similarly, $A^{\oplus n}$ is a shorthand for the element $\sum_{k=0}^{n-1}A\otimes e_k$ in the von Neumann algebra $B(H)\bar{\otimes}l_{\infty}(\mathbb{Z}).$

Hardy-Littlewood submjorization is defined by the formula
$$S\prec\prec T\mbox{ iff }\int_0^t\mu(s,S)ds\leq\int_0^t\mu(s,T)ds,\quad t>0,$$
where we use the identification of the singulvar value sequence with the corresponding step function.

\begin{lem}\label{diagonal majorization lemma}  If $(p_k)_{k\in\mathbb{Z}^d}$ is a sequence of pairwise orthogonal projections, then
$$\bigoplus_{k\in\mathbb{Z}^d}p_kTp_k\prec\prec T.$$
\end{lem}	

\begin{fact}\label{2infty majorization fact} If $T\in\mathcal{L}_{2,\infty}$ and if $S\prec\prec T,$ then $S\in\mathcal{L}_{2,\infty}$ and
$$\|S\|_{2,\infty}\leq 2\|T\|_{2,\infty}.$$
\end{fact}
\begin{proof} For every $t>0,$ we have
$$t\mu(t,S)\leq\int_0^t\mu(s,S)ds\leq\int_0^t\mu(s,T)ds\leq\|T\|_{2,\infty}\int_0^ts^{-\frac12}ds=2t^{\frac12}\|T\|_{2,\infty}.$$
Dividing by $t^{\frac12}$ and taking the supremum over $t>0,$ we complete the proof.
\end{proof}

\begin{fact}\label{l2infty triange inequality} We have
$$\|A+B\|_{2,\infty}\leq 2^{\frac12}\|A\|_{2,\infty}+2^{\frac12}\|B\|_{2,\infty}.$$
\end{fact}
\begin{proof} For every $t>0,$ we have
$$\mu(t,A+B)\leq\mu(\frac{t}{2},A)+\mu(\frac{t}{2},B).$$
Thus,
$$\|A+B\|_{2,\infty}\leq \sup_{t>0}t^{\frac12}(\mu(\frac{t}{2},A)+\mu(\frac{t}{2},B))=$$
$$=2^{\frac12}\sup_{t>0}t^{\frac12}(\mu(t,A)+\mu(t,B))\leq 2^{\frac12}\|A\|_{2,\infty}+2^{\frac12}\|B\|_{2,\infty}.$$
\end{proof}

\begin{fact}\label{trivial fact} If $A\in B(H),$ then 
$$\|A^{\oplus n}\|_{2,\infty}\geq n^{\frac12}\|A\|_{\infty}.$$
\end{fact}
\begin{proof} Clearly,
$$\mu(A^{\oplus n})=\sigma_n\mu(A)\geq\sigma_n(\|A\|_{\infty}\chi_{(0,1)})=\|A\|_{\infty}\chi_{(0,n)}.$$	
\end{proof}

In the next lemma, we estimate the product of the operator $(1-\Delta_{\mathbb{R}^d})^{\frac{d}{4}+\frac12}$ with the commutator $\Big[M_{\phi},(1-\Delta_{\mathbb{R}^d})^{-\frac{d}{4}}\Big].$

\begin{lem}\label{psdo commutator lemma} If $\phi\in C^{\infty}_c(\mathbb{R}^d),$ then the operator
$$(1-\Delta_{\mathbb{R}^d})^{\frac{d}{4}+\frac12}\Big[M_{\phi},(1-\Delta_{\mathbb{R}^d})^{-\frac{d}{4}}\Big]$$
extends to a bounded operator.
\end{lem}
\begin{proof} The operator $(1-\Delta_{\mathbb{R}^d})^{-\frac{d}{4}}$ is a pseudo-differential operator of order $-\frac{d}{2}.$ The operator $M_{\phi}$ is a pseudo-differential operator of order $0.$ By Theorem 2.5.1 in \cite{RuzhanskyTurunen}, the operator
$$\Big[M_{\phi},(1-\Delta_{\mathbb{R}^d})^{-\frac{d}{4}}\Big]$$
is a pseudo-differential operator of order $-\frac{d}{2}-1.$ Consequently, the operator
$$(1-\Delta_{\mathbb{R}^d})^{\frac{d}{4}+\frac12}\Big[M_{\phi},(1-\Delta_{\mathbb{R}^d})^{-\frac{d}{4}}\Big]$$
is a pseudo-differential operator of order $0.$ By Theorem 2.4.2 in \cite{RuzhanskyTurunen}, it is bounded.
\end{proof}

\subsection{Proof of Theorem \ref{bad rd theorem}}

The following proposition is the key to the proof of Theorem \ref{bad rd theorem}.  It provides a concrete example of the function for which the estimate in Theorem \ref{bad rd theorem} is supposed to fail. It delivers the estimate of the expression on the left hand side of Theorem \ref{bad rd theorem} from below. 

\begin{prop}\label{main bad prop} If
$$f_n=\sum_{k\in\{0,\cdots,n-1\}^d}\chi_{k+\frac1n\mathbb{B}^d},\quad n\in\mathbb{N},$$
then there exists a constant $c_d'$ such that
$$n^{\frac{d}{2}}\|M_{\chi_{\frac1n\mathbb{B}^d}}(1-\Delta_{\mathbb{R}^d})^{-\frac{d}{4}}\|_{\infty}\leq   2^{\frac32}\Big\|M_{f_n}(1-\Delta_{\mathbb{R}^d})^{-\frac{d}{4}}\Big\|_{2,\infty}+c_d',\quad n\in\mathbb{N}.$$
\end{prop}
\begin{proof} Let $K=[-\frac12,\frac12]^d$ and let $p_k=M_{\chi_{k+K}},$ $k\in\mathbb{Z}^d.$ Applying Lemma \ref{diagonal majorization lemma}, we obtain
$$\bigoplus_{k\in\mathbb{Z}^d}M_{\chi_{k+K}}M_{f_n}(1-\Delta_{\mathbb{R}^d})^{-\frac{d}{4}}M_{\chi_{k+K}}\prec\prec M_{f_n}(1-\Delta_{\mathbb{R}^d})^{-\frac{d}{4}}.$$
For $n\geq 2,$ we have
$$M_{\chi_{k+K}}M_{f_n}=M_{\chi_{k+\frac1n\mathbb{B}^d}}.$$
For $n\geq2,$ we infer from the Scholium \ref{2infty majorization fact} with
$$T=M_{f_n}(1-\Delta)^{-\frac{d}{4}},$$
that
$$2\Big\|M_{f_n}(1-\Delta_{\mathbb{R}^d})^{-\frac{d}{4}}\Big\|_{2,\infty}\geq \Big\|\bigoplus_{k\in\{0,\cdots,n-1\}^d}M_{\chi_{k+\frac1n\mathbb{B}^d}}(1-\Delta_{\mathbb{R}^d})^{-\frac{d}{4}}M_{\chi_{k+K}}\Big\|_{2,\infty}.$$
	
Let $\phi\in C^{\infty}_c(\mathbb{R}^d)$ be supported in $K$ and such that $\phi=1$ on $\frac12K$ with $\|\phi\|_{\infty}=1.$ Let $\phi_k(t)=\phi(t-k),$ $t\in\mathbb{R}^d.$ It follows that
\begin{equation}\label{mbp eq1}
2\Big\|M_{f_n}(1-\Delta_{\mathbb{R}^d})^{-\frac{d}{4}}\Big\|_{2,\infty}\geq\Big\|\bigoplus_{k\in\{0,\cdots,n-1\}^d}M_{\chi_{k+\frac1n\mathbb{B}^d}}(1-\Delta_{\mathbb{R}^d})^{-\frac{d}{4}}M_{\phi_k}\Big\|_{2,\infty}.
\end{equation}
	
For $n\geq4,$ we have
$$M_{\chi_{k+\frac1n\mathbb{B}^d}}=M_{\chi_{k+\frac1n\mathbb{B}^d}}M_{\phi_k}.$$
Therefore,
$$\bigoplus_{k\in\{0,\cdots,n-1\}^d}M_{\chi_{k+\frac1n\mathbb{B}^d}}(1-\Delta_{\mathbb{R}^d})^{-\frac{d}{4}}=\bigoplus_{k\in\{0,\cdots,n-1\}^d}M_{\chi_{k+\frac1n\mathbb{B}^d}}(1-\Delta_{\mathbb{R}^d})^{-\frac{d}{4}}M_{\phi_k}+$$
$$+\bigoplus_{k\in\{0,\cdots,n-1\}^d}M_{\chi_{k+\frac1n\mathbb{B}^d}}\Big[M_{\phi_k},(1-\Delta_{\mathbb{R}^d})^{-\frac{d}{4}}\Big].$$
	
It follows from the quasi-triangle inequality (see Scholium \ref{l2infty triange inequality}) that 
$$\|A+B\|_{2,\infty}\leq 2^{\frac12}\|A\|_{2,\infty}+2^{\frac12}\|B\|_2.$$
Consequently,
\begin{align*}
\Big\|\bigoplus_{k\in\{0,\cdots,n-1\}^d}M_{\chi_{k+\frac1n\mathbb{B}^d}}(1-\Delta_{\mathbb{R}^d})^{-\frac{d}{4}}&\Big\|_{2,\infty}\leq 2^{\frac12}\Big\|\bigoplus_{k\in\{0,\cdots,n-1\}^d}M_{\chi_{k+\frac1n\mathbb{B}^d}}(1-\Delta_{\mathbb{R}^d})^{-\frac{d}{4}}M_{\phi_k}\Big\|_{2,\infty}\\&
+2^{\frac12}\Big\|\bigoplus_{k\in\{0,\cdots,n-1\}^d}M_{\chi_{k+\frac1n\mathbb{B}^d}}\Big[M_{\phi_k},(1-\Delta_{\mathbb{R}^d})^{-\frac{d}{4}}\Big]\Big\|_2.
\end{align*}

Using \eqref{mbp eq1}, we obtain
\begin{align}\label{hren}
\Big\|\bigoplus_{k\in\{0,\cdots,n-1\}^d}M_{\chi_{k+\frac1n\mathbb{B}^d}}(1-\Delta_{\mathbb{R}^d})^{-\frac{d}{4}}\Big\|_{2,\infty}\leq 2^{\frac32}\Big\|M_{f_n}(1-\Delta_{\mathbb{R}^d})^{-\frac{d}{4}}\Big\|_{2,\infty}\nonumber\\
+\Big(2\sum_{k\in\{0,\cdots,n-1\}^d}\Big\|M_{\chi_{k+\frac1n\mathbb{B}^d}}\Big[M_{\phi_k},(1-\Delta_{\mathbb{R}^d})^{-\frac{d}{4}}\Big]\Big\|_2^2\Big)^{\frac12}.
\end{align}
	
Now, we estimate the second summand on the right hand side of \eqref{hren}.
\begin{align*}
&\sum_{k\in\{0,\cdots,n-1\}^d}\Big\|M_{\chi_{k+\frac1n\mathbb{B}^d}}\Big[M_{\phi_k},(1-\Delta_{\mathbb{R}^d})^{-\frac{d}{4}}\Big]\Big\|_2^2\\&
\leq \sum_{k\in\{0,\cdots,n-1\}^d}\Big\|M_{\chi_{k+\frac1n\mathbb{B}^d}}(1-\Delta_{\mathbb{R}^d})^{-\frac{d}{4}-\frac12}\Big\|_2^2\cdot\Big\|(1-\Delta_{\mathbb{R}^d})^{\frac{d}{4}+\frac12}\Big[M_{\phi_k},(1-\Delta_{\mathbb{R}^d})^{-\frac{d}{4}}\Big]\Big\|_{\infty}^2\\&
=\Big(\sum_{k\in\{0,\cdots,n-1\}^d}\Big\|M_{\chi_{k+\frac1n\mathbb{B}^d}}(1-\Delta_{\mathbb{R}^d})^{-\frac{d}{4}-\frac12}\Big\|_2^2\Big)\cdot\Big\|(1-\Delta_{\mathbb{R}^d})^{\frac{d}{4}+\frac12}\Big[M_{\phi},(1-\Delta_{\mathbb{R}^d})^{-\frac{d}{4}}\Big]\Big\|_{\infty}^2\\&
=\Big\|M_{f_n}(1-\Delta_{\mathbb{R}^d})^{-\frac{d}{4}-\frac12}\Big\|_2^2\cdot\Big\|(1-\Delta_{\mathbb{R}^d})^{\frac{d}{4}+\frac12}\Big[M_{\phi},(1-\Delta_{\mathbb{R}^d})^{-\frac{d}{4}}\Big]\Big\|_{\infty}^2\stackrel{L.\ref{psdo commutator lemma}}{=}\frac12(c_d')^2.
\end{align*}

To estimate (from below) the left hand side of \eqref{hren}, it remains to note that the operators
$$\Big\{M_{\chi_{k+\frac1n\mathbb{B}^d}}(1-\Delta_{\mathbb{R}^d})^{-\frac{d}{4}}\Big\}_{k\in\{0,\cdots,n-1\}^d}$$
are pairwise unitarily equivalent (with the help of shift operator) and, thus,
$$\Big\|\bigoplus_{k\in\{0,\cdots,n-1\}^d}M_{\chi_{k+\frac1n\mathbb{B}^d}}(1-\Delta_{\mathbb{R}^d})^{-\frac{d}{4}}\Big\|_{2,\infty}\stackrel{F.\ref{trivial fact}}{\geq} n^{\frac{d}{2}}\Big\|M_{\chi_{\frac1n\mathbb{B}^d}}(1-\Delta_{\mathbb{R}^d})^{-\frac{d}{4}}\Big\|_{\infty}.$$
\end{proof}

The following important result is proved in \cite{RSZ}. Here,
$$\psi(t)=
\begin{cases}
\frac1{\log(\frac{e}{t})},& t\in(0,1)\\
t,& t\geq 1
\end{cases}
$$
and $\mathtt{M}_{\psi}$ is the corresponding Marcinkiewicz space (see \cite{KPS}).

\begin{prop}\label{space boundedness optimal} Let $d\in\mathbb{N}.$ Let $f=\mu(f)\in\mathtt{M}_{\psi}(0,\infty).$ We have
$$\Big\|(1-\Delta_{\mathbb{R}^d})^{-\frac{d}{4}}M_{f\circ r_d}(1-\Delta_{\mathbb{R}^d})^{-\frac{d}{4}}\Big\|_{\infty}\geq c_d\|f\|_{\mathtt{M}_{\psi}}.$$
\end{prop}

\begin{proof}[Proof of Theorem \ref{bad rd theorem}] Let $f_n$ be as in Proposition \ref{main bad prop}. We have
$$n^{\frac{d}{2}}\|M_{\chi_{\frac1n\mathbb{B}^d}}(1-\Delta_{\mathbb{R}^d})^{-\frac{d}{4}}\|_{\infty}\leq   2^{\frac32}\Big\|M_{f_n}(1-\Delta_{\mathbb{R}^d})^{-\frac{d}{4}}\Big\|_{2,\infty}+c_d',\quad n\in\mathbb{N}.$$
By Proposition \ref{space boundedness optimal}, we have
$$\|M_{\chi_{\frac1n\mathbb{B}^d}}(1-\Delta_{\mathbb{R}^d})^{-\frac{d}{4}}\|_{\infty}=\|(1-\Delta_{\mathbb{R}^d})^{-\frac{d}{4}}M_{\chi_{\frac1n\mathbb{B}^d}}(1-\Delta_{\mathbb{R}^d})^{-\frac{d}{4}}\|_{\infty}^{\frac12}$$	
$$\geq c_d^{\frac12}\|\chi_{(0,n^{-d})}\|_{\mathtt{M}_{\psi}}^{\frac12}\geq d^{\frac12}c_d^{\frac12} n^{-\frac{d}{2}}\log^{\frac12}(n),\quad n\in\mathbb{N}.$$
A combination of these inequalities yields
$$2^{\frac32}\Big\|M_{f_n}(1-\Delta_{\mathbb{R}^d})^{-\frac{d}{4}}\Big\|_{2,\infty}\geq d^{\frac12}c_d^{\frac12} \log^{\frac12}(n)-c_d',\quad n\in\mathbb{N}.$$
Consequently,
$$8\Big\|(1-\Delta_{\mathbb{R}^d})^{-\frac{d}{4}}M_{f_n}(1-\Delta_{\mathbb{R}^d})^{-\frac{d}{4}}\Big\|_{1,\infty}\geq \big(d^{\frac12}c_d^{\frac12} \log^{\frac12}(n)-c_d'\big)_+^2,\quad n\in\mathbb{N}.$$
	
We have
$$\mu(f_n)=\chi_{(0,{\rm Vol}(\mathbb{B}^d))}\mbox{ and }\|f_n\|_E=\|\chi_{(0,{\rm Vol}(\mathbb{B}^d))}\|_E,\quad n\in\mathbb{N},$$
for every symmetric quasi-Banach function space $E.$ Suppose \eqref{optimistic cwikel estimate} holds for $f_n.$ It follows that
$$\big(d^{\frac12}c_d^{\frac12} \log^{\frac12}(n)-c_d'\big)_+^2\leq 8\|\chi_{(0,{\rm Vol}(\mathbb{B}^d))}\|_E,\quad n\in\mathbb{N},$$
which is impossible.
\end{proof}

\appendix

\section{Equivalent description of the norm in Theorem \ref{best available rd cwikel}}\label{app section}

In this appendix, we simplify the expressions used in the proof of Theorem \ref{best available rd cwikel}. The argument extends the one in Theorem 3.1 of \cite{SharPLMS}.

\begin{lem} We have
$$\|Vf\|_{L_M(\mathbb{B}^d)}\leq (2d+2)\|f\|_{L_M(\mathbb{R}^d\backslash\mathbb{B}^d)}+(2d+2)\int_{\mathbb{R}^d\backslash\mathbb{B}^d}|f(s)|\log(1+|s|)ds.$$
\end{lem}
\begin{proof} Without loss of generality, $f\geq0.$ Suppose that
$$\int_{\mathbb{R}^d\backslash\mathbb{B}^d}f(s)\log(1+|s|)ds\leq 1,\quad \int_{\mathbb{R}^d\backslash\mathbb{B}^d}M(f(s))ds\leq 1.$$
Since $M(t)\geq t,$ $t>0,$ it follows that
$$\int_{\mathbb{R}^d\backslash\mathbb{B}^d}f(s)ds\leq 1.$$

It follows that
\begin{align*}
\int_{\mathbb{B}^d}M((Vf)(u))du&=\int_{\mathbb{B}^d}M\Big(|u|^{-2d}f(\frac{u}{|u|^2})\Big)du=\int_{\mathbb{R}^d\backslash\mathbb{B}^d}M(|s|^{2d}f(s))|s|^{-2d}ds\\&
=\int_{\mathbb{R}^d\backslash\mathbb{B}^d}f(s)\cdot \log(e+|s|^{2d}f(s))ds.
\end{align*}
	
We have
$$e+ab\leq e+eab\leq e(1+a)(1+b).$$
Thus,
$$\log(e+|s|^{2d}f(s))\leq 1+\log(1+|s|^{2d})+\log(1+f(s))\leq$$
$$\leq 1+2d\log(1+|s|)+\log(e+f(s)).$$
Thus,
\begin{align*}
\int_{\mathbb{B}^d}M((Vf)(t))dt&\leq \int_{\mathbb{R}^d\backslash\mathbb{B}^d}f(s)ds+2d\int_{\mathbb{R}^d\backslash\mathbb{B}^d}f(s)\log(1+|s|)ds\\&
+\int_{\mathbb{R}^d\backslash\mathbb{B}^d}M(f(s))ds\leq 1+2d+1=2d+2.
\end{align*}
\end{proof}

\begin{lem} We have
$$\|f\|_{L_M(\mathbb{R}^d\backslash\mathbb{B}^d)}\leq \|Vf\|_{L_M(\mathbb{B}^d)}.$$
\end{lem}
\begin{proof} Denote, for brevity, $g=Vf$ and note that $f=Vg.$ Without loss of generality, $f\geq0.$ Suppose that
$$\int_{\mathbb{B}^d}M(g(s))ds\leq 1.$$
It follows that
\begin{align*}
\int_{\mathbb{R}^d\backslash\mathbb{B}^d}& M(f(u))du=\int_{\mathbb{R}^d\backslash\mathbb{B}^d}M\Big(|u|^{-2d}g(\frac{u}{|u|^2})\Big)du=\int_{\mathbb{B}^d}M(|s|^{2d}g(s))|s|^{-2d}ds\\ &
=\int_{\mathbb{R}^d\backslash\mathbb{B}^d}g(s)\cdot \log(e+|s|^{2d}g(s))ds\leq \int_{\mathbb{R}^d\backslash\mathbb{B}^d}g(s)\cdot \log(e+g(s))ds\leq 1.
\end{align*}
\end{proof}

\begin{lem} We have
$$\int_{\mathbb{R}^d\backslash\mathbb{B}^d}|f(s)|\log(1+|s|)ds\leq c_d\|Vf\|_{L_M(\mathbb{B}^d)}.$$
\end{lem}
\begin{proof} Denote, for brevity, $g=Vf$ and note that $f=Vg.$ Thus,
\begin{align*}
\int_{\mathbb{R}^d\backslash\mathbb{B}^d}|f(s)|\log(1+|s|)ds&=\int_{\mathbb{R}^d\backslash\mathbb{B}^d}|s|^{-2d}|g(\frac{s}{|s|^2})|\log(1+|s|)ds\\&
=\int_{\mathbb{B}^d}|g(u)|\log(1+\frac1{|u|})du=\int_{\mathbb{B}^d}|g(u)||h(u)|du,
\end{align*}
where $h(u)=\log(1+\frac1{|u|}),$ $u\in\mathbb{B}^d.$ It follows that
$$\int_{\mathbb{R}^d\backslash\mathbb{B}^d}|f(s)|\log(1+|s|)ds\leq\int_0^{{\rm Vol}(\mathbb{B}^d)}\mu(t,g)\mu(t,h)dt.$$
Obviously,
$$\mu(t,h)=\log\big(1+\big(\frac{t}{{\rm Vol}(\mathbb{B}^d)}\big)^{-\frac1d}\big),\quad 0<t<{\rm Vol}(\mathbb{B}^d).$$
Thus,
$$\int_{\mathbb{R}^d\backslash\mathbb{B}^d}|f(s)|\log(1+|s|)ds\leq \int_0^{{\rm Vol}(\mathbb{B}^d)}\mu(t,g)\log\big(1+\big(\frac{t}{{\rm Vol}(\mathbb{B}^d)}\big)^{-\frac1d}\big)dt.$$
It is immediate that
$$\int_0^{{\rm Vol}(\mathbb{B}^d)}\mu(t,g)\log\big(1+\big(\frac{t}{{\rm Vol}(\mathbb{B}^d)}\big)^{-\frac1d}\big)dt\leq c_d\int_0^1\mu(t,g)\log(\frac1t)dt.$$
Hence,
$$\int_{\mathbb{R}^d\backslash\mathbb{B}^d}|f(s)|\log(1+|s|)ds\leq c_d\int_0^1\mu(t,g)(1+\log_+(\frac1t))dt.$$
The right hand side is the norm $\|g\|_{\Lambda_1},$ where $\Lambda_1$ is the Lorentz space featuring in \cite{LSZ-last-kalton}. Since the Orlicz space $L_M$ coincides with the Lorentz space $\Lambda_1,$ the assertion follows.
\end{proof}

\begin{proof}[Proof of Proposition \ref{equivalence prop}] The assertion follows by combining three lemmas above. 
\end{proof}

\section{Proof of Theorem \ref{best available rd cwikel} for $d=2$}\label{frank section}

This appendix contains a short proof of Theorem \ref{best available rd cwikel} for $d=2.$ The proof was provided to us by Professor Frank and is presented here with his kind permission.

For a (possibly unbounded) self-adjoint operator $S,$ we denote by $N(I,S)$ the number of eigenvalues of $S$ in the interval $I.$ The latter number set to be $+\infty$ if spectrum of the operator $S$ on $I$ is not discrete.

The proof is based on the main result in \cite{SharPLMS} which can be read as follows.

\begin{thm}\label{shar thm} Let $d=2$ and let $0\leq f\in L_M(\mathbb{R}^2).$ We have
$$N((-\infty,0),-\Delta_{\mathbb{R}^2}-M_f)\leq 1+c_2\Big(\|f\|_{L_M(\mathbb{R}^2)}+\int_{\mathbb{R}^2}|f(s)|\log(1+|s|)ds\Big).$$
\end{thm}

Strictly speaking, the right hand side in \cite{SharPLMS} is written as
$$1+\|f\|_{L_{\mathcal{B}}(\mathbb{R}^2)}+\int_{\mathbb{R}^2}|f(s)|\log(1+|s|)ds,$$
where
$$\mathcal{B}(t)=(1+t)\log(1+t)-t,\quad t>0.$$
This quantity is equivalent to the one in the right hand side of the theorem above since Orlicz functions $M$ and $\mathcal{B}$ are equivalent for large values of $t.$

Spectral estimates for Schr\"odinger operators and for Cwikel operators are related via Birman-Schwinger principle. An abstract version of Birman-Schwinger principle suitable for our purposes can be found e.g. in Proposition 2.3 in \cite{PushJFA} or in Lemma 1.4 in \cite{BS_Novgorod}.

\begin{thm} Let $T$ be positive and boundedly invertible. Let $V$ be positive and bounded. Suppose that $V^{\frac12}T^{-\frac12}$ is compact. It follows that
$$N((-\infty,0),T-V)=N((1,\infty),T^{-\frac12}VT^{-\frac12}).$$
\end{thm}

We are now ready to prove the main result in the appendix.

\begin{proof}[Proof of Theorem \ref{best available rd cwikel} for $d=2$] We may assume without loss of generality that $f\geq0$ is bounded and compactly supported. The approximation argument required to prove the assertion in full generality repeats the one in the proof of Theorem 1.3 {\it mutatis mutandi}.

Let $t>0.$

By Theorem 2.3 in \cite{LSZ-last-kalton}, we have 
$$\Big\|(1-\Delta_{\mathbb{R}^2})^{-\frac12}M_f(1-\Delta_{\mathbb{R}^2})^{-\frac12}\Big\|_{\infty}\leq c_1\|f\|_{L_M(\mathbb{R}^2)}.$$
A somewhat weaker bound, which is, however, sufficient for the proof of Theorem \ref{best available rd cwikel}, can also be directly deduced from  \cite{Solomyak1994,SharPLMS}. Therefore,
$$N((t,\infty),(1-\Delta_{\mathbb{R}^2})^{-\frac12}M_f(1-\Delta_{\mathbb{R}^2})^{-\frac12})=0$$
whenever
$$t>c_1\|f\|_{L_M(\mathbb{R}^2)}.$$

Suppose now that
$$t\leq c_1\|f\|_{L_M(\mathbb{R}^2)}.$$
By Birman-Schwinger principle and Theorem \ref{shar thm}, we have
$$N((t,\infty),(1-\Delta_{\mathbb{R}^2})^{-\frac12}M_f(1-\Delta_{\mathbb{R}^2})^{-\frac12})=N((1,\infty),(1-\Delta_{\mathbb{R}^2})^{-\frac12}M_{t^{-1}f}(1-\Delta_{\mathbb{R}^2})^{-\frac12})=$$
$$=N((-\infty,0),1-\Delta_{\mathbb{R}^d}-M_{t^{-1}f})=N((-\infty,-1),-\Delta_{\mathbb{R}^d}-M_{t^{-1}f})\leq$$
$$\leq N((-\infty,0),-\Delta_{\mathbb{R}^d}-M_{t^{-1}f})\leq$$
$$\leq 1+\frac{c_2}{t}\Big(\|f\|_{L_M(\mathbb{R}^2)}+\int_{\mathbb{R}^2}|f(s)|\log(1+|s|)ds\Big).$$
By the assumption on $t,$ we have
$$1\leq \frac{c_1}{t}\|f\|_{L_M(\mathbb{R}^2)}.$$
It follows that
$$N((t,\infty),(1-\Delta_{\mathbb{R}^2})^{-\frac12}M_f(1-\Delta_{\mathbb{R}^2})^{-\frac12})\leq$$
$$\leq \frac{c_1+c_2}{t}\Big(\|f\|_{L_M(\mathbb{R}^2)}+\int_{\mathbb{R}^2}|f(s)|\log(1+|s|)ds\Big).$$

Combining the estimates in the preceding paragraphs, we obtain
$$N((t,\infty),(1-\Delta_{\mathbb{R}^2})^{-\frac12}M_f(1-\Delta_{\mathbb{R}^2})^{-\frac12})\leq$$
$$\leq \frac{c_1+c_2}{t}\Big(\|f\|_{L_M(\mathbb{R}^2)}+\int_{\mathbb{R}^2}|f(s)|\log(1+|s|)ds\Big),\quad t>0.$$
In other words,
$$\Big\|(1-\Delta_{\mathbb{R}^2})^{-\frac12}M_f(1-\Delta_{\mathbb{R}^2})^{-\frac12}\Big\|_{1,\infty}\leq$$
$$\leq (c_1+c_2)\Big(\|f\|_{L_M(\mathbb{R}^2)}+\int_{\mathbb{R}^2}|f(s)|\log(1+|s|)ds\Big).$$
\end{proof}

\end{document}